\documentclass[twoside, 11pt]{article}

\usepackage{graphicx, subfigure}
\usepackage[top=1in,bottom=1in,left=1in,right=1in,footnotesep=0.5in]{geometry}
\usepackage[sort&compress, numbers]{natbib} 
\usepackage{amsfonts, amsmath, amssymb, amsthm}
\usepackage{MnSymbol}
\usepackage{mathtools}
\mathtoolsset{showonlyrefs}
\usepackage{enumitem}
\usepackage{comment}

\usepackage{color}
\definecolor{darkred}{RGB}{100,0,0}
\definecolor{darkgreen}{RGB}{0,100,0}
\definecolor{darkblue}{RGB}{0,0,150}

\usepackage{hyperref}
\hypersetup{colorlinks=true, linkcolor=darkred, citecolor=darkgreen, urlcolor=darkblue}
\usepackage{url}

\def\haus{{\sf d}_{\rm H}}
\def\scl{\textsc{l}}
\def\sck{\textsc{k}}
\def\ring{A}
\def\supp{{\rm supp}}

\def\level{\cL}
\def\up{\cU}

\def\d{{\rm d}}
\def\ms{{\sf MS}}

\def\ball{B}

\newtheorem{thm}{Theorem}[section]
\newtheorem{theorem}{Theorem}[section]

\newtheorem{lem}[thm]{Lemma}
\newtheorem{lemma}[thm]{Lemma}

\theoremstyle{remark}

\newtheorem{remark}[thm]{Remark}

\newtheorem{property}{Property}

\def\beq{\begin{equation}} 
\def\eeq{\end{equation}}
\def\beqn{\begin{eqnarray*}}
\def\eeqn{\end{eqnarray*}}
\def\Bitem{\begin{itemize}\setlength{\itemsep}{.2in}}
\def\bitem{\begin{itemize}\setlength{\itemsep}{.05in}}
\def\eitem{\end{itemize}}
\def\Benum{\begin{enumerate}\setlength{\itemsep}{.2in}}
\def\benum{\begin{enumerate}\setlength{\itemsep}{.05in}}
\def\eenum{\end{enumerate}}
\def\bmult{\begin{multline*}}
\def\emult{\end{multline*}}
\def\bcenter{\begin{center}}
\def\ecenter{\end{center}}
\def\bframe{\begin{frame}}
\def\eframe{\end{frame}}

\newcommand{\thmref}[1]{Theorem~\ref{thm:#1}}

\newcommand{\propertyref}[1]{Property~\ref{property:#1}}

\newcommand{\lemref}[1]{Lemma~\ref{lem:#1}}
\newcommand{\secref}[1]{Section~\ref{sec:#1}}

\DeclareMathOperator*{\argmax}{arg\, max}
\DeclareMathOperator*{\argmin}{arg\, min}

\def\cA{\mathcal{A}}
\def\cB{\mathcal{B}}
\def\cC{\mathcal{C}}

\def\cF{\mathcal{F}}

\def\cH{\mathcal{H}}

\def\cL{\mathcal{L}}
\def\cM{\mathcal{M}}

\def\cS{\mathcal{S}}
\def\cT{\mathcal{T}}
\def\cU{\mathcal{U}}
\def\cV{\mathcal{V}}

\def\cY{\mathcal{Y}}
\def\cZ{\mathcal{Z}}

\def\bbR{\mathbb{R}}

\def\eps{\varepsilon}

\def\1{\mathbbm{1}}

\definecolor{purple}{rgb}{0.4,.1,.9}

\begin{document}
\thispagestyle{empty}

\title{Clustering by Hill-Climbing: Consistency Results}
\author{
Ery Arias-Castro\footnote{University of California, San Diego, California, USA (\url{https://math.ucsd.edu/\~eariasca/})} 
\and 
Wanli Qiao\footnote{George Mason University, Fairfax, Virginia, USA (\url{https://mason.gmu.edu/\~wqiao/})}
}
\date{}
\maketitle

\begin{abstract}
We consider several hill-climbing approaches to clustering as formulated by \citet{fukunaga1975} in the 1970's. We study both continuous-space and discrete-space (i.e., medoid) variants and establish their consistency. 

\end{abstract}

\section{Introduction} 
\label{sec:introduction}

Clustering methods based on `climbing' the density landscape date back to the 1970's, in particular, to work by K.~Fukunaga and his collaborators.\footnote{~Hill-climbing often refers to greedy approach to optimizing an objective function. Such strategies have been suggested in the context of clustering, including the algorithm of \citet{kernighan1970efficient}, the $k$-means algorithms proposed by \citet{lloyd1982least} and \citet{hartigan1979ak}, and even the EM algorithm of \citet{dempster1977maximum} when used to fit a mixture distribution. In the present paper, we reserve this term for approaches that climb the landscape defined by a density.}
Indeed, in 1975, \citet{fukunaga1975} proposed to ``assign each observation to the nearest mode along the direction of the gradient". Formally, the gradient ascent line starting at a point $x$ is the curve given by the image of $\gamma_x$, the parameterized curve defined by the the following ordinary differential equation (ODE)
\begin{equation} \label{gradient flow}
\gamma_x(0) = x; \quad
\dot\gamma_x(t) = \nabla f(\gamma_x(t)), \quad t \ge 0.
\end{equation} 
In this fashion, a point $x$ is assigned to the (critical) point at the end of the gradient line above, meaning $\gamma_x(\infty)$. 
\citet{fukunaga1975} then added that ``To accomplish this, one could move each observation a small step in the direction of the gradient and iteratively repeat the process on the transformed observations until tight clusters result near the modes." This led them to propose what is known in numerical analysis as a forward Euler scheme (with step size being $\rho > 0$ here):
\begin{equation} \label{euler}
x(0) = x; \quad
x(k+1) = x(k) + \rho \nabla f(x(k)), \quad k \ge 0.
\end{equation}
Under some standard conditions, the scheme is consistent in that, as $\rho \to 0$, the sequence converges in an appropriate sense to the gradient line $\gamma_x$. 
We will refer to this scheme as the {\em Euler Shift}.

In practice, the density needs to be estimated, and this is often done by kernel density estimation. This was already considered in \cite{fukunaga1975}, and the name `mean shift' comes from the fact that when using a kernel density estimator, the gradient of that estimator is proportional to the shift in mean with respect to another kernel --- what \citet{cheng1995} called a `shadow' of the kernel used to estimate the density.
In the actual implementation proposed by \citet{fukunaga1975} --- which is nowadays known as {\em Blurring Mean Shift} --- at each iteration all the sample points are moved and the density estimate is recomputed based on the new locations. \citet{cheng1995} contrasted this with what he calls {\em Mean Shift}, where instead the density is estimated based on the original sample before the sample points are moved by mean shift as described above. 
We note that in both implementations the points are moved by successive mean shifts, which differs from applying the Euler scheme above. 

The approach to clustering advocated by \citet{fukunaga1975} has generated a good amount of enthusiasm over the past decades, leading to a number of methods. 
\citet{carreira2015clustering} provides a fairly recent review of this work.
Among them, arguably the simplest variant is what we call {\em Max Shift}, where at each step the location is changed to a point in the neighborhood with largest density value:
\begin{equation} \label{max shift density}
x(0) = x_0; \quad
x(k+1) \in \argmax_{x \in \bar\ball(x(k), \eps)} f(x), \quad k \ge 0.
\end{equation}
Ties are broken in an arbitrary but deterministic way (for simplicity), and if $f(x(k+1)) = f(x(k))$, the process stops.
The parameter $\eps > 0$ defines the size of neighborhood around the present location where the maximization takes places to compute the next location in the sequence.
This is effectively embedded in a method proposed by \citet{chazal2013persistence}, which they called {\em ToMATo} (for Topological Mode Analysis Tool) --- although the overall approach is more sophisticated and also includes some merging of attraction basins based on (topological) persistence considerations.

In the present paper, we establish the consistency of Euler Shift, Mean Shift, Max Shift, and a few other variants (including a regularized version of the one proposed in \cite{koontz1976graph}) in a concise and comprehensive manner. To be clear, consistency refers to the task of clustering in the sense of \citet{fukunaga1975}, where points are grouped according to the critical points where the density gradient ascent flow \eqref{gradient flow} leads them.
By doing this, we contribute to the building of a mathematical foundation for this type of clustering methods, and adds to existing work in that area.
For one, Euler Shift is known to be consistent \cite{arias2016estimation, comaniciu2002mean, carreira2000mode}, and there is a surrounding literature on the problem of estimating the gradient lines of a density \cite{cheng2004estimating} and even density ridges \cite{genovese2014nonparametric, qiao2016theoretical}.
Further, via the shadow kernel concept of \citet{cheng1995}, Mean Shift can be directly related to Euler Shift and thus proven to be also consistent.

The remainder of the paper is organized as follows.
In \secref{setting}, we introduce the framework and some concepts and notation.
In \secref{prototype}, we introduce a prototypical hill-climbing algorithm and establish its consistency. We then specialize this to several variants. 
In \secref{methods}, we establish the large-sample consistency of the corresponding methods.
Throughout the paper, we distinguish between {\em algorithm} --- defined based on a given density function --- and {\em method} --- an algorithm applied to an estimate of the density based on an iid sample from that density.
In \secref{medoids}, we consider medoid variants of the previous algorithms where the sequence is constrained to be made of a given set of points (the sample points in practice).
We discuss some open problems in \secref{discussion} and gather some additional technical results in the \hyperref[sec:auxiliary]{Appendix}.

\section{Setting}
\label{sec:setting}
We lay in this section the foundations, starting with our assumptions on the underlying density --- which are very standard for this literature --- and introducing some basic concepts regarding the gradient flow it defines. We also discuss the role of the point set $\cY$ that appears in \eqref{max slope shift} and \eqref{max shift}.

\subsection{The density $f$}
\label{sec:density}
Throughout, we consider a density with respect to the Lebesgue measure on $\bbR^d$ --- denoted $f$ everywhere --- assumed to satisfy the following conditions:
\begin{itemize}
\item {\em Zero at infinity.}
$f$ converges to zero at infinity, meaning, $f(x) \to 0$ as $\|x\| \to \infty$.
\item {\em Twice differentiable.} 
$f$ is twice differentiable everywhere with bounded and uniformly continuous zeroth, first, and second derivatives.
\item 
{\em Non-degenerate critical points.} 
The Hessian is non-singular at every critical point of $f$.
\end{itemize}
The first condition is equivalent to $f$ having bounded upper level sets, meaning that $\up_s := \{f \ge s\}$ is bounded, and therefore compact by continuity of $f$, for any $s > 0$.
The second condition is simply a smoothness assumption on the density. (Note that the uniform continuity of the zeroth and first derivatives is implied by the fact that the second derivative exits everywhere and is bounded.)
And a function that satisfies the third condition is sometimes referred to as a {\em Morse function} \citep{milnor1963morse}.
Similar conditions are standard in the literature cited in the \hyperref[sec:introduction]{Introduction}. 

Define
\begin{align}
\label{kappa}
\kappa_0 = \sup_x f(x), \quad
\kappa_1 = \sup_x \|\nabla f(x)\|, \quad 
\kappa_2 = \sup_x \|\nabla^2 f(x)\|.
\end{align}
Then $f$ is $\kappa_1$-Lipschitz, meaning
\begin{align}
\label{kappa1}
|f(y) - f(x)| \le \kappa_1 \|y-x\|, \quad \forall x, y,
\end{align}
and $\nabla f$ is $\kappa_2$-Lipschitz, meaning
\begin{align}
\label{kappa2}
\|\nabla f(y) - \nabla f(x)\| \le \kappa_2 \|y-x\|, \quad \forall x, y,
\end{align}
and the following Taylor expansion holds
\begin{align}
\label{taylor2}
\big|f(y) - f(x) - \nabla f(x)^\top (y-x)\big|
\le \tfrac12 \kappa_2 \|y-x\|^2, \quad \forall x, y.
\end{align}

We will denote $N(x) := \nabla f(x)/\|\nabla f(x)\|$, which is well-defined whenever $\nabla f(x) \ne 0$, and in fact differentiable, with derivative equal to 
\begin{align}
\label{DN}
DN(x) = \frac{\nabla^2 f(x)}{\|\nabla f(x)\|} - \frac{\nabla f(x) \nabla f(x)^\top \nabla^2 f(x)}{\|\nabla f(x)\|^2}. 
\end{align}

In most of the paper, we will assume that the density is available. In practice, of course, it needs to be estimated, and this is most often done via kernel density estimation. Using some stability results implying roughly speaking that things do not change much if the estimation is accurate enough, we will port the consistency results established for the setting where the density is known to the setting where it is estimated.
We differentiate between {\em algorithm}, which is applied with knowledge of the density; and {\em method}, which is applied without knowledge of the density except for an estimate of the density, typically derived from a sample. This terminology may not be standard, but we find it useful in the confined setting of the present work.

By {\em population} we mean the support of the density, which we will denote by $\supp f$ on occasion. When we talk of a sample, we will assume it to be generated iid from $f$.

\subsection{Gradient lines and gradient flow}
\label{sec:setting gradient}
Under the above conditions on $f$, $\nabla f$ is Lipschitz, and this is enough for standard theory for ODEs \citep[Sec 17.1]{hirsch2012differential} to justify the definition in \eqref{gradient flow} of the gradient ascent line $\gamma_x$ originating at any point $x$. This, and the fact that this is a gradient flow \citep[Sec 9.3]{hirsch2012differential}, gives the following.

\begin{lemma}
\label{lem:gradient flow}
For any $x$, the function $\gamma_x$ defined in \eqref{gradient flow} is well-defined on $[0,\infty)$, with $\gamma_x(t)$ converging to a critical point of $f$ as $t \to \infty$.
\end{lemma}

The {\em basin of attraction} of a point $x_*$ is defined as $\{x: \gamma_x(\infty) = x_*\}$. Note that this set is empty unless $x_*$ is a critical point (i.e., $\nabla f(x_*) = 0$).
In the gradient line view of clustering, we call a cluster any basin of attraction of a mode.
It turns out that, if $f$ is a Morse function \citep{milnor1963morse}, then all these basins of attraction, sometimes called stable manifolds, provide a partition of the support up to a set of zero measure.

\begin{lemma}
\label{lem:basins}
Under the assumed regularity conditions, the basins of attraction of the local maxima, by themselves, cover the population, except for a set of zero measure.
\end{lemma}
Indeed, by \lemref{gradient flow}, the basins of attraction partition the entire population. In addition, the set of critical points is discrete \citep[Cor~3.3]{banyaga2013lectures}, the basin of attraction of each critical point that is not a local maximum is a (differentiable) submanifold of co-dimension at least one\footnote{ The requirement in that theorem that the function be compactly supported is clearly not essential.} \cite[Th~4.2]{banyaga2013lectures}, and therefore has zero Lebesgue measure.
For more background on Morse functions and their use in statistics, see the recent articles of \citet{chacon2015population} and \citet{chen2017statistical}.

The following lemma gives the continuity of the gradient flow curve with respect to the Hausdorff metric when seen as a subset of $\bbR^d$ indexed by the starting point, which can be of independent interest --- its proof is given in the \hyperref[sec:auxiliary]{Appendix}. We note that this result is stronger than the well-known continuity of trajectories of gradient flows with respect to the starting points \emph{up to a fixed time point}. See, for example, the main theorem in \citep[Sec 17.3]{hirsch2012differential}. Let $\haus$ denote the Hausdorff metric. 
\begin{lemma}
\label{lem:full continuity}
The gradient ascent flow, seen here as the function $x \mapsto \gamma_x([0, \infty))$ mapping $(\bbR^d, \|\cdot\|)$ to $(2^{\bbR^d}, \haus)$, is continuous in the basin of attraction of a mode. 
\end{lemma}

For the sake of clarity, we will sometimes work with the gradient line parameterized by arc length. Equivalently, for a point $x$, this means considering the gradient flow of $N$, or more explicitly, 
\begin{equation} \label{zeta}
\zeta_x(0) = x; \quad
\dot\zeta_x(t) 
= N(\zeta_x(t)).
\end{equation}
Note that $\zeta_x$ is only defined over $[0,\ell_x]$, where $\ell_x$ is the length of the gradient ascent line originating at $x$ (meaning the length of $\gamma_x$).
In view of our assumptions on $f$, 
assuming $x$ is not a critical point, $\zeta_x$ is twice continuously differentiable on $[0,\ell_x)$ with
\begin{align}
\label{zeta2} 
\ddot \zeta_x(t) 
= DN(\zeta_x(t)) \dot \zeta_x(t), \quad \text{with $DN$ given in \eqref{DN}.}
\end{align}

\subsection{Level sets}
For a positive real number $s > 0$, the $s$-level set of $f$ is given by
\begin{equation}
\label{level}
\level_s := \{x : f(x) = s\}.
\end{equation}
while the $s$-upper level set of $f$ is given by
\begin{equation}
\label{up}
\up_s := \{x : f(x) \ge s\}.
\end{equation}
Throughout, whether specified or not, we will only consider levels that are in $(0, \kappa_0)$. Note that, because $f$ converges to zero at infinity and is continuous, its (upper) level sets are compact.
We call any connected component of an upper level set a {\em level cluster}. 
This is in congruence with the level set definition of cluster offered by \citet{hartigan1975clustering}.
Hartigan also defined what is now known as the {\em cluster tree}, which is the partial ordering between clusters that comes with the set inclusion operation: indeed, when two level clusters intersect, one of them must contain the other.
We say that a level cluster $\cC$ is a {\em leaf (level) cluster} if the cluster tree does not branch out past $\cC$, or said differently, if all the descendants of $\cC$ have at most one child. Note that the last descendant of a leaf cluster is a singleton defined by a mode.
For a point $x$ and $s \le f(x)$, let $\cC_s(x)$ denote the $s$-level cluster that contains $x$, and $\cC(x)$ will be shorthand for $\cC_{f(x)}(x)$. 

There is a large amount of literature on the estimation of level sets and on the estimation of the cluster tree, but in the present paper we will only use some basic results, including the following.

\begin{lemma}
\label{lem:cluster mode}
Any level cluster contains at least one mode. Moreover, a mode coincides with the intersection of all the level clusters that contain that mode.
\end{lemma}

The following is a slightly different version of \cite[Lem 5.9]{arias2021asymptotic}.
\begin{lemma}
\label{lem:cluster ball}
Let $x_*$ be a mode of $f$ and let $s_* = f(x_*)$. Then there is a constant $\delta>0$ such that $C \equiv C_\delta = \sup_{x\in\bar\ball(x_*,\delta)}\max\{\sqrt{\lambda_{\max}(x)},2/\sqrt{\lambda_{\min}(x)}\}$ is finite, where $\lambda_{\max}(x)$ and $\lambda_{\min}(x)$ are the largest and smallest eigenvalues of $-\nabla^2 f(x)$, respectively; and 
\begin{align}
\label{Ct_bounds}
\bar\ball(x_*, \tfrac1C \sqrt{s_* - s}) \subset \cC_s \subset \ball(x_*, C \sqrt{s_* - s}), \quad \text{for all }  s\in(s_* - \delta^2/C^2, s_*).
\end{align}
\end{lemma}


\section{Consistency: Algorithms}
\label{sec:algorithms}

In this section we look at various hill-climbing clustering algorithms. As we indicated in the Introduction, an algorithm is defined based on an available density. 

We adopt in this paper the definition of clustering proposed by \citet{fukunaga1975}, where we ``assign each [point] to the nearest mode along the direction of the gradient". That is, we assign a point $x$ to $\gamma_x(\infty)$, where $\gamma_x$ is the gradient ascent line originating from $x$, defined in \eqref{gradient flow}, and $\gamma_x(\infty) := \lim_{t \to \infty} \gamma_x(t)$ is the endpoint where that line terminates. Consequently, the population --- meaning the support of the density --- is partitioned according to the basins of attraction of the density critical points.
As just discussed in \secref{setting gradient}, this definition is justified, and even though not all the critical points are modes, it is true by \lemref{basins} that the basins of attraction associated with modes are the ones that truly matter.

With the definition of clustering that we espouse here, we say that an algorithm is {\em consistent} if it moves almost any point $x$ in the support to $\gamma_x(\infty)$ when the neighborhood size is small enough. We make this more precise below.


\subsection{Prototype}
\label{sec:prototype} 

We start by discussing a prototypical hill-climbing algorithm that we prove to be consistent. We then show that a number of hill-climbing algorithms satisfy the same core properties, implying that these algorithms are consistent.

The prototypical algorithm that we consider, when initialized at some point in the support of the density, say $x_0$, produces a sequence, denoted $(x_k)$. The core properties we just alluded to are the following:

\begin{property}[\em The shifts are of comparable size]
\label{property:1}
For some positive function $S$, for all $k$, except perhaps for the last shift,
\begin{equation}
\label{shift size}
\eps S(\|\nabla f(x_k)\|)
\le \|x_{k+1} - x_k\| 
\le \eps.
\end{equation}
The function $S$ will be taken to be non-decreasing without loss of generality. The quantity $\eps$ can be made small by appropriately tuning the algorithm. 
\end{property}  

\begin{property}[\em The process converges to a mode when initialized in the vicinity of that mode]
\label{property:2}
Suppose that $x_*$ is a mode. Then there is $\delta > 0$ such that, for $\eps > 0$ small enough, if $x_0 \in \ball(x_*, \delta)$ then $(x_k)$ converges to $x_*$.
\end{property}

\begin{property}[\em The shifts are close to the gradient at the corresponding location]
\label{property:3}
For all shifts, except perhaps for the last shift,
\begin{equation}
\label{shift angle}
x_{k+1} - x_k
= \|x_{k+1} - x_k\| N(x_k) \pm \|x_{k+1} - x_k\| R(\|x_{k+1} - x_k\|, \|\nabla f(x_k)\|), 
\end{equation}
whenever $\nabla f(x_k) \ne 0$, where $(u,v) \mapsto R(u, v)$ is a continuous function that is increasing in $u$ and decreasing in $v$, and satisfies $R(0, v) = 0$ for all $v>0$. 
\end{property}
We note that the functions $S$ and $R$ may depend on the starting point $x_0$.

\begin{theorem}
\label{thm:prototype}
Consider an algorithm that satisfies the above properties.
Let $x_*$ denote a mode. Then for any point $x_0$ in the basin of attraction of $x_*$, when initialized at $x_0$ and with $\eps$ made small enough, the algorithm produces a sequence that converges to $x_*$.
\end{theorem}

\begin{proof}
Let $\delta_* > 0$ and $\eps_* > 0$ be as in \propertyref{2}, so that it suffices to show that the sequence that the algorithm with parameter $\eps$ small enough that $\eps \le \eps_*$ constructs produces a sequence, denoted $(x_k)$ henceforth, that reaches $\ball(x_*, \delta_*)$.
Let $\zeta$ be shorthand for $\zeta_{x_0}$, defined in \eqref{zeta}, and let $\ell$ de shorthand for $\ell_{x_0}$. 
We let $\cZ_t := \zeta([0,t])$, which is the gradient line up to time $t$. (Note that `time' represents length with the chosen parameterization of the gradient line.)
Since $\cZ_\ell$ joins $x_0$ and $x_*$, the gradient line certainly enters that ball, and the idea is to show that $(x_k)$ remains close to that curve, at least until entering that ball. 

Let $t_\# := \inf\{t \ge 0 : \|\zeta(t) - x_*\| = \delta_*/3\}$. Then 
define
$
\nu := \frac12 \min\{\|\nabla f(z)\| : z \in \cZ_{t_\#}\},
$
and note that $\nu > 0$ by the fact that $\|\nabla f\|$ is continuous and (strictly) positive on $\cZ_{t_\#}$ because the gradient line $\cZ_\ell$ does not contain a critical point other than $x_*$ at its very end. 
By an application of \eqref{kappa2}, we have that $\|\nabla f(y)\| \ge \nu$ for all $y$ in the `tube' $\cT := \ball(\cZ_{t_\#}, \delta_{\rm tube})$, where $\delta_{\rm tube} :=  \nu/\kappa_2$.

{\em The sequence $(z_k)$.}
Define the sequence $t_0 := 0$ and $z_0 := 0$, and for $k \ge 1$, $t_k := t_{k-1} + \eps_k$, where $\eps_k : =\|x_k-x_{k-1}\|$, and $z_k := \zeta(t_k)$.
Of course, as the discretization gets finer and finer, the sequence $(z_k)$ gets closer and closer to the gradient ascent line, and the basic idea is to compare the sequence $(x_k)$ to the sequence $(z_k)$. 
Let $k_\# := \max\{k : t_k \le t_\#\}$, and note that, 
since $t_{k_\#+1} > t_\#$ and 
\begin{align}
\label{k lower bound}
t_{k_\#+1} = \sum_{k=1}^{k_\#+1} \eps_k \le (k_\#+1) \eps, \quad \text{we have $k_\# > t_\#/\eps -1$.}  
\end{align}
Letting $z_\# := \zeta(t_\#)$, we have that $\|z_\# - x_*\| = \delta_*/3$ by construction, and also
\begin{align}
\|z_\# - z_{k_\#}\| 
= \|\zeta(t_\#) - \zeta(t_{k_\#})\| 
= t_\# - t_{k_\#} 
< t_{k_\#+1} - t_{k_\#} 
\le \eps.
\end{align}
Assuming $\eps$ is small enough that $\eps < \delta_*/3$, we can guarantee that $\|z_{k_\#} - x_*\| \le 2\delta_*/3$.
Also, a Taylor expansion gives, 
\begin{align}
z_k - z_{k-1}
&=  \zeta(t_k) - \zeta(t_{k-1}) \\
&= (t_k - t_{k-1}) \dot\zeta(t_{k-1}) \pm \tfrac12 \sup_{0 \le t \le t_k} \|\ddot\zeta(t)\| (t_k-t_{k-1})^2.
\end{align}
While $(t_k - t_{k-1}) \dot\zeta(t_{k-1}) = \eps_k N(z_{k-1})$, based on \eqref{zeta2} and \eqref{DN}, and the fact that $\|\nabla f(z_{k-1})\| \ge 2 \nu$ for any $k \le k_\#$,
\begin{align}
\sup_{0 \le t \le t_k} \|\ddot\zeta(t)\|
\le \sup_{0 \le t \le t_\#} \|D N(\zeta(t))\| 
\le \frac{\kappa_2}{2\nu} + \kappa_2.
\end{align}
Hence,
\begin{align}
z_k - z_{k-1}
&= \eps_k N(z_{k-1}) \pm C_1 \eps^2 , \quad \text{for all } 1 \le k \le k_\#. \label{z diff}
\end{align}

{\em The sequence $(x_k)$.}
Define $d_k := \|x_k - z_k\|$. 
We bound $d_k$ by induction for $0 \le k \le k_\#$. 
Note that $d_0 = 0$ since $z_0 = x_0$. 
Recall the function $R$ in \eqref{shift angle}, and define $Q_1(\eps) := R(\eps, \nu) + C_1 \eps$.
Let $C_2 := \kappa_2/\nu+\kappa_2$, and note that, due to \eqref{DN}, $\|DN(x)\| \le C_2$ for all $x \in \cT$. 
Recall the function $S$ appearing in \eqref{shift size} and take $\eps$ small enough that
\begin{align}
\label{eps small1}
Q_2(\eps) := Q_1(\eps) \frac{\exp[C_2 t_\#/S(\nu)] - 1}{C_2} \le \delta_{\rm tube},
\end{align}
which is possible because $Q_1(\eps) \to 0$ as $\eps \to 0$.
We are now ready to set the induction hypothesis: suppose that
\begin{align}
\label{induction1}
d_m \le Q_1(\eps) \frac{\exp[C_2 \eps m] - 1}{C_2} \le \delta_{\rm tube}, \quad \forall m = 1, \dots, k-1.
\end{align}
This is certainly true at $k = 1$, which primes our induction.
Note that the two inequalities are part of the induction.
We now bound $d_k$, assuming that $k \le k_\#$.
Since $z_{k-1} = \zeta(t_{k-1})$ with $t_{k-1} \le t_\#$, and $\|x_{k-1} - z_{k-1}\| = d_{k-1} \le \delta_{\rm tube}$ by induction, we have $\|\nabla f(x_{k_1})\| \ge \nu$. With that, and \propertyref{3} (specifically \eqref{shift angle}), we derive
\begin{align}
x_k - x_{k-1}
&= \eps_k N(x_{k-1}) \pm \eps_k R(\eps_k, \|\nabla f(x_{k-1})\|) \\ 
&= \eps_k N(x_{k-1}) \pm \eps R(\eps, \nu).
\label{q diff}
\end{align}
The bound \eqref{q diff} combined with \eqref{z diff} gives
\begin{align}
x_k - z_k
&= x_k - x_{k-1} + x_{k-1} - z_{k-1} + z_{k-1} - z_k \\
&= \eps_k N(x_{k-1}) \pm \eps R(\eps, \nu) + x_{k-1} - z_{k-1} - \eps_k N(z_{k-1}) \pm C_1 \eps^2,
\end{align}
which after applying the triangle inequality, results in
\begin{align}
d_k
&\le d_{k-1} + \eps_k \|N(x_{k-1}) - N(z_{k-1})\| + \eps R(\eps, \nu) + C_1 \eps^2.
\end{align}
Since the segment $[z_{k-1},x_{k-1}]$ is inside the ball $\ball(z_{k-1}, \delta_{\rm tube})$ (because, again, $d_{k-1} \le \delta_{\rm tube}$), and that ball is inside $\cT$ (because, again, $z_{k-1} = \zeta(t_{k-1})$ with $t_{k-1} \le t_\#$), $N$ is Lipschitz with constant $C_2$ inside $\ball(z_{k-1}, \delta_{\rm tube})$, which then gives
\begin{align}
\|N(x_{k-1}) - N(z_{k-1})\| 
\le C_2 \|x_{k-1} - z_{k-1}\|.
\end{align}
Hence, we have
\begin{align}
d_k
&\le d_{k-1} + \eps C_2 \|x_{k-1} - z_{k-1}\| + \eps R(\eps, \nu) + C_1 \eps^2\\
&= (1+ C_2 \eps) d_{k-1} + \eps Q_1(\eps). \label{d_bound}
\end{align}
Now, calling in the inequality \eqref{induction1} at $m = k-1$, and simplifying using the fact that $e^a-1-a \ge 0$ for all $a$, we deduce
\begin{align}
d_k 
\le Q_1(\eps) \frac{\exp[C_2 \eps k] - 1}{C_2}.
\end{align}
This is only the first inequality that we needed to propagate.
We now turn to the second one, which consists in bounding the right-hand side by $\delta_{\rm tube}$.
Since we are considering $k \le k_\#$, we have $t_k \le t_\#$, and by using \propertyref{1} (specifically \eqref{shift size}), we further get
\begin{align}\label{prototype property 1}
t_\# \ge t_k = \sum_{m=1}^k \eps_m \ge \eps \sum_{m=1}^k S(\|\nabla f(x_{m-1})\|.
\end{align}
For $m = 1, \dots, k-1$, $d_m \le \delta_{\rm tube}$ by induction, implying as we already saw above that $x_m \in \cT$, in turn implying that $\|\nabla f(x_{m-1})\| \ge \nu$. Plugging this into \eqref{prototype property 1}, we get $t_\# \ge k S(\nu) \eps$, or $k \eps \le t_\#/S(\nu)$. By monotonicity, we thus have
\begin{align}
Q_1(\eps) \frac{\exp[C_2 \eps k] - 1}{C_2}
\le Q_2(\eps)
\le \delta_{\rm tube},
\end{align}
the latter inequality being \eqref{eps small1}.
Thus the induction proceeds. We have thus established that
\begin{align}
\label{d_final}
d_k = \|x_k - z_k\| \le Q_2(\eps), \quad \forall k = 0, \dots, k_\#.
\end{align}

{\em Conclusion.}
In particular, we have $d_{k_\#} = \|x_{k_\#} - z_{k_\#}\| \le Q_2(\eps)$, and taking $\eps$ small enough that $Q_2(\eps) < \delta_*/3$, by the triangle inequality and the fact that $\|z_{k_\#} - x_*\| \le 2\delta_*/3$, we can guarantee that $\|x_{k_\#} - x_*\| < \delta_*$.
This is what we needed to prove.
\end{proof}

In this whole section we continue to use the same notation as in the proof above, except that we make the dependence on the starting point $x$ explicit whenever needed as in, e.g., $\nu(x)$ denoting $\nu$ when associated with $x$. 

Next we provide a uniform version of \thmref{prototype}, in the sense that with $\eps$ small enough, the result in \thmref{prototype} holds for almost all the starting points in $\cA$, which is the union of the basins of attraction of all the modes. Its proof uses the continuity of $\nu$, as given in \lemref{nu_continuity}. 



Recall functions $S$ and $R$ appearing in \eqref{shift size} and \eqref{shift angle}, respectively. We consider an algorithm satisfying the three properties with $S$ and $R$ that depend on $x_0$ only in a uniform way, as specified below. For any $s\in(0,\kappa_0)$, let $\cB_s$ be the union of the basins of attraction for all the critical points in $\up_s$ that are not modes. Note that if $\cB_s$ is not empty, it consists of finitely many $k$-dimensional submanifolds, where $k=0,\cdots,d-1$, and has zero Lebesgue measure,  as indicated right below \lemref{basins}. For $s,\delta> 0$, define $\Gamma_{\delta,s} = \up_s \bigcap \ball(\cB_s,\delta)^{\complement}$. For any $s,\delta> 0$ such that $\Gamma_{\delta,s}$ is not empty, we assume that there exist $S\equiv S_{\delta,s}$ and $R\equiv R_{\delta,s}$ such that the properties hold for all $x_0\in\Gamma_{\delta,s}$. This is the case, for example, of Max Shift; see \lemref{max shift level cluster} and also \lemref{max shift step size} for example.

\begin{theorem}
\label{thm:prototype_uniform}
Consider the prototypical algorithm satisfying the above properties and the assumption on the uniformity of $R$ and $S$. For every $\eta>0$, there exists an $\eps_0>0$ and a measurable set $\Omega_\eta$ with probability measure at least $1 - \eta$ such that for all $\eps\in(0,\eps_0]$, the algorithm applied to any $x_0 \in \Omega_\eta$ returns the associated mode, meaning, $\lim_{t\to\infty}\gamma_{x_0}(t)$.
\end{theorem}

\begin{proof}
  


For an arbitrarily small but fixed $\eta>0$, let $s_\eta$ be the largest $s>0$ such that the probability measure of $\up_s^{\complement}$ is not larger than $\eta/4$. Define $\Gamma_{\delta,s}^*:=\up_{s}\bigcap\ball(\cB_{s},\delta).$ Note that  as $\delta \to 0$, the Lebesgue measure of $\Gamma_{\delta,s_\eta}^*$ is of order $O(\delta)=o(1)$, and hence the probability measure of $\Gamma_{\delta,s_\eta}^*$ is also of order $o(1)$. Let $\delta_\eta$ be the largest $\delta>0$ such that the probability measure of $\Gamma_{\delta,s_\eta}^*$ is not larger than $\eta/4$. Then the probability measure of $\Gamma_{\delta_\eta,s_\eta}=:\Omega_\eta$ is at least $1-\eta$. For simplicity, we denote $S= S_{\delta_\eta,s_\eta}$ and $R= R_{\delta_\eta,s_\eta}$.

Let $\cC$ be the union of balls $\ball(x_*,\delta_*/3)$, where $\delta_*$ is as in \propertyref{2} and depends on the mode $x_*$. Let $\Gamma_{\eta}^- = \Omega_\eta\setminus \cC$. 
Define $\underline\nu_\eta: = \inf_{x\in \Gamma_{\eta}^-} \nu(x)$, which is positive, since $\Gamma_{\eta}^-$ is a compact set, and $\nu$ is continuous and positive on $\cA\supset\Gamma_{\eta}^-$ by \lemref{nu_continuity}. 
Based on the proof of \thmref{prototype}, especially \eqref{eps small1}, in order to guarantee that the prototypical algorithm returns the correct mode for all $x\in \Gamma_{\eta}^-$, we only need to choose $\eps>0$ small enough that 
\begin{align}
\label{uniform_condition}
R(\eps,\underline\nu_\eta) + \kappa_2\eps \le \frac{\kappa_2}{\exp[\kappa_0(\kappa_2/\underline\nu_\eta + \kappa_2)/(\underline\nu_\eta S(\underline\nu_\eta))] - 1}.
\end{align}
Note that here for $t_\#$ appearing in \eqref{eps small1}, we have used $t_\#(x) \le \kappa_0 / \nu(x)$, which is implied by the following inequalities:
\begin{align}
\kappa_0 \ge f(\zeta_x(t_\#(x))) - f(x) = \int_0^{t_\#(x)} \nabla f(\zeta_x(\tau)) \dot\zeta_x(\tau) \d\tau = \int_0^{t_\#(x)} \|\nabla f(\zeta_x(\tau)) \| \d\tau \ge \nu(x) t_\#(x).
\end{align}
Because $R(\eps,\underline\nu_\eta) + \kappa_2\eps \searrow 0$ as $\eps\searrow0$, we indeed have \eqref{uniform_condition} for $\eps$ small enough, 
and thus have shown that when $\eps$ is small enough, for any starting point $x_0\in\Omega_\eta$, we always have the correct clustering result using the prototype algorithm. 
\end{proof}

\subsection{Max Shift and Max Slope Shift}
\label{sec:max (slope) shift}
In this subsection, we show that Max Shift \eqref{max shift density} and a related approach proposed early on by \citet*{koontz1976graph} which we call {\em Max Slope Shift} \eqref{max slope shift density}, are consistent by showing that they both satisfy the three properties listed in \secref{prototype}.

\subsubsection{Max Shift}
\label{sec:max shift}

Max Shift --- introduced in \eqref{max shift density} --- is arguably the simplest, and thus most prototypical, hill-climbing clustering algorithm. We show here that it satisfies the properties required of the prototypical algorithm of \secref{prototype}, establishing its consistency as the neighborhood size tends to zero ($\eps \to 0$). We do so in a series of lemmas.

\begin{lemma}
\label{lem:max shift level cluster}
Take any point $x_0$ in the population. If $\eps$ is smaller than the minimum separation between $\cC(x_0)$ and any other level cluster at the same level, Max Shift initialized at $x_0$ converges (in a finite number of steps) to a mode belonging to $\cC(x_0)$.
\end{lemma}

It could be the case that $f(x_0) = 0$, but because $x_0 \in \supp f$, it must be that for any $\eps>0$ there is $x \in \ball(x_0, \eps)$ such that $f(x) > 0$. Therefore, the sequence does not stay at $x_0$, and even though $\cC(x_0)$ is, in this case, a connected component of $\supp f$, the lemma does say that the sequence converges to some density mode.

\begin{proof}
If $x_0$ is a mode, then by choosing $\eps$ small enough that it is a maximum inside $\ball(x_0, \eps)$, the sequence immediately ends at $x_\infty = x_0$. Therefore, in the remaining of the proof, we consider a point $x_0$ which is not a mode.

Since $\eps$ is smaller than the minimum separation between $\cC(x_0)$ and any other $f(x_0)$-level cluster, Max Shift initialized at $x_0$ outputs a sequence that must remain in $\cC(x_0)$. Otherwise, there must be $k$ such that $x_k \equiv x(k) \in \cC(x_0)$ and $x_{k+1} \notin \cC(x_0)$, and because $\|x_{k+1} - x_k\| \le \eps$ and the separation between $\cC(x_0)$ and any other connected component of $\up_{f(x_0)}$ exceeds $\eps$, it must be the case that $x_{k+1} \notin \up_{f(x_0)}$, triggering $f(x_{k+1}) < f(x_0) \le f(x_k)$, contradicting the rules governing the algorithm which makes it hill-climbing.

We now show that the sequence converges. We just saw that the sequence is inside $\cC(x_0)$, which is compact. Therefore, the sequence has at least one accumulation point inside $\cC(x_0)$. 
Let $x_\ddag$ be such a point. By the fact that the sequence of density values $(f(x_k))$ is increasing, it must be the case that $f(x_\ddag) \ge f(x_k)$ for all $k$, and not just the $k$'s indexing the subsequence converging to $x_\ddag$. 
Take $k_\ddag$ such that $x_{k_\ddag}$ is within distance $\eps$ from $x_\ddag$. If $x_{k_\ddag+1} \ne x_\ddag$, it must mean that $f(x_{k_\ddag+1}) \ge f(x_\ddag)$, which then implies that $f(x_{k_\ddag+1}) \ge f(x_k)$ for all $k$, itself implying that the sequence stops at, and thus converges to, $x_{k_\ddag+1}$.

We have thus established that the sequence converges to some point, say $x_\infty$, inside $\cC(x_0)$. And that the convergence happens in a finite number of steps: as soon as $\|x_k - x_\infty\| \le \eps$, it must be that $x_{k+1} = x_\infty$. This in turn implies that $x_\infty$ is a mode within $\ball(x_\infty, \eps)$ --- although perhaps not the only one. Indeed, let $k_\infty$ denote the step at which the sequence stops at $x_\infty$. By the rules governing the algorithm, it stops there because there are no points within that ball with a strictly higher density value.
\end{proof}

\begin{lemma}
\label{lem:max shift step size}
Let $(x_k)$ denote the Max Shift sequence originating from some arbitrary point $x_0$ with $t_0 := f(x_0) > 0$. Suppose that $\eps$ is small enough that \lemref{max shift level cluster} applies, and small enough that any mode in $\cC(x_0)$ is a maximum within a radius of $\eps$. Then, at each step $k$, the shift $x_{k+1}-x_k$ is of size exactly $\eps$, except possibly for the very last shift.
In particular, Max Shift satisfies \propertyref{1}.
\end{lemma}

\begin{proof}
Suppose that, for some $k$, $\eps_{k+1} := \|x_{k+1} - x_k\| < \eps$. By how Max Shift constructs the sequence, this implies that $x_{k+1}$ is a maximum in $\ball(x_k, \eps)$, and in particular a maximum in $\ball(x_{k+1}, \eps - \eps_{k+1})$, and therefore a mode. Because \lemref{max shift level cluster} applies, $x_{k+1} \in \cC(x_0)$, and so the sequence must terminate at $x_{k+1}$ because, by assumption, $x_{k+1}$ is maximum in $\ball(x_{k+1}, \eps)$.
\end{proof}

\begin{lemma}
\label{lem:max shift mode}
Suppose that $x_*$ is a mode. Then there is $\delta > 0$ such that, for $\eps > 0$ small enough, Max Shift initialized at any point in $\ball(x_*, \delta)$ converges to $x_*$.
In particular, Max Shift satisfies \propertyref{2}.
\end{lemma}

\begin{proof}
Our assumptions on the density imply that the critical points are isolated. Therefore, there is $\delta_\ddag>0$ such that $x_*$ is the only critical point in $\ball(x_*, \delta_\ddag)$. Let $s_\ddag = \max\{f(x) : x \in \partial\ball(x_*, \delta_\ddag)\}$. By construction, $s_\ddag < s_* := f(x_*)$ and $\cC_{s_\ddag}(x_*) \subset \ball(x_*, \delta_\ddag)$. By \lemref{cluster ball}, there is $\delta \le \delta_\ddag$ such that $\ball(x_*, \delta) \subset \cC_{s_\ddag}(x_*)$. Suppose $\eps>0$ is smaller than the separation between $\cC_{s_\ddag}(x_*)$ and any other $s_\ddag$-level cluster.

Now take a starting point $x_0 \in \ball(x_*, \delta)$, and let $s_0 := f(x_0)$. Since $x_0 \in \cC_{s_\ddag}(x_*)$, we have $t_0 \ge s_\ddag$, and therefore we must have $\cC(x_0) \subset \cC_{s_\ddag}(x_*)$. 
Note that the separation between $\cC(x_0)$ and any other $s_0$-level cluster must exceed $\eps$. This is because any $s_0$-level cluster must be inside a $s_\ddag$-level cluster, and $\cC_{s_\ddag}(x_*)$ cannot contain more than one $s_0$-level cluster because of \lemref{cluster mode} and the fact that it only contains one mode.
We are thus able to apply \lemref{max shift level cluster} to assert that the sequence converges to a local mode within $\cC_{s_\ddag}(x_*)$, which must be $x_*$ since, again, $\cC_{s_\ddag}(x_*)$ does not contain any other mode by construction.
\end{proof}

\begin{lemma}
\label{lem:max shift angle}
Let $(x_k)$ denote the Max Shift sequence originating from some arbitrary point $x_0$ in the population. 
Assume that $\eps$ is small enough that \lemref{max shift step size} applies.
Then for each shift, except perhaps for the last one, 
\begin{equation}
\label{max shift angle}
x_{k+1} - x_k
= \eps N(x_k) \pm \frac{C_0 \eps^{3/2}}{\|\nabla f(x_k)\|^{1/2}},
\end{equation}
whenever $\nabla f(x_k) \ne 0$.
In particular, Max Shift satisfies \propertyref{3}.
\end{lemma}
\begin{remark}
We note that if three times continuous differentiability of $f$ is assumed, the $\eps^{3/2}$ in the remainder of \eqref{max shift angle} can be enhanced to $\eps^2$ by performing higher order Taylor expansions in \eqref{taylor2} and \eqref{f_diff}.
\end{remark}

\begin{proof}
Let $u_{k+1} := x_{k+1} - x_k$ and $\eps_{k+1} := \|u_{k+1}\|$.
The result comes from comparing $f(x_{k+1})$, which by construction maximizes $f(x)$ over $x \in \bar\ball(x_k, \eps)$, with $f(x'_{k+1})$ where $x'_{k+1} : =x_k + \eps_{k+1} N(x_k)$. Since by construction we must have $\eps_{k+1} \le \eps$, we have $x'_{k+1}\in \bar\ball(x_k, \eps)$, triggering $f(x_{k+1}) \ge f(x'_{k+1})$.

In general, using \eqref{taylor2}, we have
\begin{align}
\label{f_diff}
f(x+u) - f(x+\|u\| N(x))
&= \nabla f(x)^\top u - \|u\| \|\nabla f(x)\| \pm \kappa_2 \|u\|^2,
\end{align}
so that 
\begin{align}
f(x+u) - f(x+\|u\| N(x)) < 0 \quad \text{if} \quad \frac{\nabla f(x)^\top u}{\|\nabla f(x)\| \|u\|} < 1 - \frac{\kappa_2 \|u\|}{\|\nabla f(x)\|}.
    \end{align}
Applying this to $x_k, x_{k+1}, u_{k+1}$ as defined by the Max Shift, we obtain  
\begin{align}
\label{max shift angle proof1}
N(x_k)^\top u_{k+1} \ge (1 - \kappa_2\eps/\|\nabla f(x_k)\|) \eps_{k+1}.
\end{align} 

First, suppose that $\eps/\|\nabla f(x_k)\| > 1/\kappa_2$.
Then, taking $C_0$ large enough that \smash{$C_0/\kappa_2^{1/2} \ge 2$}, we have
\[\|x_{k+1}-x_k\| - \eps N(x_k) = \pm 2 \eps \quad \text{and} \quad C_0 \eps^{3/2}/\|\nabla f(x_k)\|^{1/2} > \eps C_0/\kappa_2^{1/2} \ge 2\eps,\]
so that \eqref{max shift angle} holds.
Next, assume that $\eps/\|\nabla f(x_k)\| \le 1/\kappa_2$.
Set $u_{k+1} = a N(x_k) + v_{k+1}$ where $a := N(x_k)^\top u_{k+1}$ and $v_{k+1} \perp N(x_k)$. 
Using \lemref{max shift step size} and \eqref{max shift angle proof1}, we get that, except possibly for the last shift, $a \ge (1 - \kappa_2\eps/\|\nabla f(x_k)\|) \eps \ge 0$. This and Pythagoras gives 
\begin{align}
\label{max shift angle proof2}
\|v_{k+1}\|^2 = \eps_{k+1}^2 - a^2 \le \eps^2 - (1 - \kappa_2\eps/\|\nabla f(x_k)\|)^2 \eps^2 \le 2 \kappa_2 \eps^3/\|\nabla f(x_k)\|,
\end{align}
so that 
\begin{align}
u_{k+1} 
&= \eps N(x_k) \pm \kappa_2\eps^2/\|\nabla f(x_k)\| \pm (2 \kappa_2 \eps^3/\|\nabla f(x_k)\|)^{1/2},
\end{align}
from which we obtain \eqref{max shift angle} after some simplification. 
\end{proof}

Because it satisfies all three properties, by \thmref{prototype}, 
\begin{center}
{\em Max Shift is consistent}.
\end{center}


\subsubsection{Max Slope Shift}
\label{sec:max slope shift}

Soon after the original paper of \citet{fukunaga1975}, \citet*{koontz1976graph} proposed a variant where, at each step, a point is moved to the point within a neighborhood window that results in the largest slope. This definition implies that a medoid algorithm which we give in \eqref{max slope shift}. 
We consider a regularized version of this algorithm that disallows shifts shorter than some set fraction of the neighborhood size. This regularization enables us to consider a continuous (medoid-less) formulation that parallels Max Shift \eqref{max shift density} quite closely, taking the following form 
\begin{equation} \label{max slope shift density}
x(0) = x_0; \quad
x(k+1) \in \begin{cases}
\text{local maxima in $\ball(x(k), \eps)$ if non-empty; otherwise} \\ 
\displaystyle\argmax_{x \in \ring(x(k), \eps, c \eps)} \frac{f(x) - f(x(k))}{\|x-x(k)\|}, \quad k \ge 0,
\end{cases}
\end{equation}
where $\ring(x, a, b) := \bar\ball(x, a) \setminus \ball(x, b)$.
The constant $0 < c < 1$ is arbitrary but fixed beforehand.
We refer to this algorithm as {\em Max Slope Shift}.

\begin{remark}
It turns out that, without such a regularization, the algorithm fails. Indeed, consider the simplest setting of a unimodal density, for example, the standard normal distribution $f(x) = \exp(-x^2/2)/\sqrt{2\pi}$. On the negative half-line, $(-\infty, 0)$, the slope is maximum at the (unique) inflection point occurring at $x = -1$. It is not hard to see that, assuming that $\eps < 1$, initialized at $x_0 \le -1$, Max Slope Shift produces a sequence that converges (in a finite number of steps) to that inflection point.
\end{remark} 

We show that Max Slope Shift satisfies the three properties, and the arguments are almost identical to those just detailed for Max Shift.

For \propertyref{1}, it is satisfied by definition since all the shifts, except possibly for the very last one, are of size in between $c \eps$ and $\eps$.

For \propertyref{2}, we can see that \lemref{max shift level cluster} applies to Max Slope Shift, simply because in the algorithm: 
{\em (i)} the steps are of size at most $\eps$;
{\em (ii)} it is hill-climbing; 
{\em (iii)} the process stops at the next step when there is a mode within distance $\eps$; 
{\em (iv)} if the process stops, it must stop at a mode.
Having verified that \lemref{max shift level cluster} applies, \lemref{max shift mode} follows immediately, so that Max Slope Shift satisfies \propertyref{2}.

For \propertyref{3}, we retrace the arguments underlying \lemref{max shift angle}. First, \eqref{max shift angle proof1} holds in exactly the same way and for exactly the same reasons. 
As detailed below that, we may focus on the situation where $\eps/\|\nabla f(x_k)\| \le 1/\kappa_2$.
We set $u_{k+1} := a N(x_k) + v_{k+1}$ as before, although this time $a \ge (1 - \kappa_2\eps/\|\nabla f(x_k)\|) \eps_{k+1} \ge 0$, which then gives 
\[\|v_{k+1}\|^2 = \eps_{k+1}^2 - a^2 \le \eps_{k+1}^2 - (1 - \kappa_2\eps/\|\nabla f(x_k)\|)^2 \eps_{k+1}^2 \le 2 \kappa_2 \eps^3/\|\nabla f(x_k)\|.\]
The last inequality being identical to the last inequality in \eqref{max shift angle proof2}, the remaining arguments apply verbatim, allowing us to conclude that Max Slope Shift satisfies \propertyref{3}.

Because it satisfies all three properties, by \thmref{prototype}, 
\begin{center}
{\em Max Slope Shift is consistent}.
\end{center}

\subsection{Euler Shift and Line Search Shift}
\label{sec:euler line search}
We now consider Euler Shift \eqref{euler} and its close variant, {\em Line Search Shift} introduced below in \eqref{line search}. 
We follow the blueprint that we detailed for the prototype algorithm of \secref{prototype}. 

\subsubsection{Euler Shift}
\label{sec:euler}
Euler Shift has already been shown to be consistent. This was most definitely done in\footnote{ An errata was issued shortly after the paper was published. Although the error was relatively minor, it was historically important as the same mistake had been made before in other works claiming to have established consistency, including \cite{comaniciu2002mean}. This was pointed out by others working in the field \citep{li2007note,ghassabeh2015sufficient}.} \cite{arias2016estimation}, following the same general proof architecture, which itself is well-known in the study of the Euler method, even at the level of textbooks as exemplified by the proof of \cite[Th 1.1]{iserles2009first}.
For completeness, we provide some details nonetheless.

For \propertyref{1}, we have
\begin{align}
\rho \|\nabla f(x_k)\| = \|x_{k+1} - x_k\| \le \rho \kappa_1,
\end{align}
so that it is satisfied with $\eps := \rho \kappa_1$.

\begin{lemma}
\label{lem:euler cluster}
For $\rho$ small enough, Euler Shift is hill-climbing, and when initialized inside a level cluster $\cC$, Euler Shift converges to a critical point inside $\cC$.
In particular, Euler Shift satisfies \propertyref{2}.
\end{lemma}

\begin{proof}
Using \eqref{taylor2}, for any $x \in \bbR^d$ and any $s > 0$, we get
\begin{align}
f(x+s \nabla f(x)) - f(x)
&= \nabla f(x)^\top (s \nabla f(x)) \pm \tfrac12 \kappa_2 \|s \nabla f(x)\|^2 \\
&\ge \big(1 - \tfrac12 \kappa_2 s\big) s \|\nabla f(x)\|^2.
\end{align}
Therefore, if $\rho > 0$ in \eqref{euler} is small enough that $\rho < 2/\kappa_2$, Euler Shift is hill-climbing. Henceforth, we assume that $0 < \rho \le 1/\kappa_2$, so that
\begin{align}
\label{euler cluster proof1}
f(x_{k+1}) - f(x_k)
&\ge \tfrac12 \rho \|\nabla f(x_k)\|^2.
\end{align}

Because it is hill-climbing, any Euler Shift sequence $(x_k)$ such that $x_0 \in \cC$ must remain in $\cC$. It must also have the property that $(f(x_k))$ converges --- since it is increasing and $f$ is assumed bounded --- and in view of \eqref{euler} this implies that $\nabla f(x_k) \to 0$. Hence, by continuity of the gradient, any accumulation point of $(x_k)$ must be a critical point. 
Furthermore, the shift size converges to zero since $\|x_{k+1} - x_k\| = \rho \|\nabla f(x_k)\|$, and because the critical points are assumed to be isolated, it must be the case (by elementary considerations) that $(x_k)$ is convergent. And by what was said earlier, the limit must be a critical point inside $\cC$.

That the algorithm satisfies \propertyref{2} follows exactly as for Max Shift, as the same arguments given in the proof of \lemref{max shift mode} apply.
\end{proof}

For \propertyref{3}, in view of \eqref{euler}, any Euler Shift sequence satisfies
\begin{align}
\label{euler property 3}
x_{k+1} - x_k = \|x_{k+1} - x_k\| N(x_k),
\end{align}
simply because the shifts are exactly aligned with the gradients at the corresponding locations.

Because it satisfies all three properties, by \thmref{prototype}, 
\begin{center}
{\em Euler Shift is consistent}.
\end{center}

\paragraph{Variants}
Euler Shift refers to the forward Euler discretization of the gradient flow of $f$. In their original proposal \cite{fukunaga1975} advocated for using, instead, the gradient flow of $\log f$, leading to the algorithm 
\begin{equation} \label{original}
x(0) = x; \quad
x(k+1) = x(k) + \rho \frac{\nabla f(x(k))}{f(x(k))}, \quad k \ge 0.
\end{equation}
The same proof blueprint applies to this variant, and others like it. Specifically, consider
\begin{equation} \label{original variant}
x(0) = x; \quad
x(k+1) = x(k) + \rho \varphi(f(x(k))) \nabla f(x(k)), \quad k \ge 0,
\end{equation}
where $\varphi$ is a non-increasing and positive function on $(0,\infty)$. Clearly, with $\varphi(a) \propto 1/a$ we recover \eqref{original}.
Showing that the algorithm in \eqref{original variant} is consistent boils down to showing that it is hill-climbing when $\rho$ is small enough.
Using \eqref{taylor2}, as before, for any $x \in \bbR^d$ and any $s > 0$, we get
\begin{align}
f(x+\rho \varphi(f(x)) \nabla f(x)) - f(x)
&= \rho \varphi(f(x)) \|\nabla f(x)\|^2 \pm \tfrac12 \kappa_2 \rho^2 \varphi(f(x))^2 \|\nabla f(x)\|^2 \\
&\ge \big(1 - \tfrac12 \kappa_2 \rho \varphi(f(x)) \big) \rho \varphi(f(x)) \|\nabla f(x)\|^2 \\
&\ge 0, \quad \text{if } \tfrac12 \kappa_2 \rho \varphi(f(x)) \le 1. \label{variant hill-climbing proof1}
\end{align}
Now, initialize the algorithm at some point $x_0$ with $f(x_0) > 0$, and let $(x_k)$ denote the sequence that results. 
Take $\rho > 0$ small enough that $\tfrac12 \kappa_2 \rho \varphi(f(x_0)) \le 1$. From \eqref{variant hill-climbing proof1}, we have that $f(x_1) \ge f(x_0)$. Suppose for induction that $(f(x_m) : m = 1, \dots, k)$ is non-decreasing. Because $\varphi$ is non-increasing and the induction hypothesis implies that $f(x_k) \ge f(x_0)$, we have $\tfrac12 \kappa_2 \rho \varphi(f(x_k)) \le \tfrac12 \kappa_2 \rho \varphi(f(x_0)) \le 1$, so that $f(x_{k+1}) \ge f(x_k)$ by \eqref{variant hill-climbing proof1}. Therefore, the induction proceeds, establishing that the algorithm is hill-climbing.

Other variants are possible. For example, in some previous work \cite{arias2021asymptotic,arias2021level} we found it useful to work with the following one
\begin{equation} \label{level shift}
x(0) = x; \quad
x(k+1) = x(k) + \rho \frac{\nabla f(x(k))}{\|\nabla f(x(k))\|^2}, \quad k \ge 0.
\end{equation}
This is a discretization of the flow 
\begin{equation}
x(0) = x; \quad
\dot x(t) = \frac{\nabla f(x(t))}{\|\nabla f(x(t))\|^2},
\end{equation}
which has the interesting property that $f(x(t)) = t$ for all (applicable) values of $t$.
(The square in the denominator is crucial, as without it we simply have the unit-speed flow given in \eqref{zeta}.)
Proving the consistency of such methods can be done by adapting the blueprint. Some adaptation is indeed necessary.  Take this particular variant, for example. It is not a priori guaranteed that it is hill-climbing all the way to a mode. However, it is easy to see that it is hill-climbing until the gradient becomes too small. Indeed, using the usual route via \eqref{taylor2}, we have that $f(x_{k+1}) \ge f(x_k)$ if $\tfrac12 \kappa_2 \rho \le \|\nabla f(x_k)\|^2$.  Although this is not enough to show that the algorithm satisfies \propertyref{2}, a quick look at the proof of \thmref{prototype} reveals that it is enough. Indeed, we care about what happens while the gradient remains $\ge \nu$, and the process is hill-climbing in that region as long as $\tfrac12 \kappa_2 \rho \le \nu^2$. \propertyref{1} is also questionable, since $\|x_{k+1} -x_k\| = \rho/\|\nabla f(x_k)\|$, but while the gradient is $\ge \nu$, we do have $\rho/\kappa_1 \le \|x_{k+1} -x_k\| \le \rho/\nu$, and this is enough.

\subsubsection{Line Search Shift}
\label{sec:line search}

In practice, a Euler scheme may not be monotone, and a way to force that is to perform a line search in the direction given by the gradient:
\begin{equation} \label{line search}
x(0) = x; \quad
x(k+1) = x(k) + \rho_k \nabla f(x(k)), \quad \rho_k \in \argmax_{r \in [0,\rho]} f(x(k) + r \nabla f(x(k))), \quad k \ge 0.
\end{equation}
We call this algorithm {\em Line Search Shift}, and we prove below that it is consistent as well.
The algorithm appears to be new, although the possibility of implementing a line search is briefly discussed in \cite{carreira2015clustering}.

For \propertyref{1}, using \eqref{taylor2}, for any $x \in \bbR^d$ and any $r \in [0, \rho]$, we have
\begin{align}
f(x+\rho \nabla f(x)) - f(x+r \nabla f(x))
&= (\rho - r) \|\nabla f(x)\|^2 \pm \tfrac12 \kappa_2 (\rho^2 + r^2) \|\nabla f(x)\|^2 \\
&\ge \big(\rho-r - \kappa_2 \rho^2\big) \|\nabla f(x)\|^2.
\end{align}
Therefore, if $\rho > 0$ in \eqref{euler} is small enough that $\kappa_2 \rho \le 1/2$, the last expression is strictly positive when $r < \rho/2$ and $\nabla f(x) \ne 0$. Hence, assuming $\rho$ is that small, in the process of running the algorithm, until reaching a critical point (at which point the sequence has converged) $\rho_k$ in \eqref{line search} must satisfy $\rho_k \in [\rho/2, \rho]$, which then implies that
\begin{align}
\label{line search property 1}
\tfrac12 \rho \|\nabla f(x_k)\| \le \rho_k \|\nabla f(x_k)\| = \|x_{k+1}-x_k\| \le \rho \kappa_1.
\end{align}
Hence, the algorithm satisfies \propertyref{1}.

Line Search Shift is hill-climbing by construction, and in view of \eqref{line search property 1}, the arguments underlying \lemref{euler cluster} apply almost verbatim to show that the algorithm satisfies \propertyref{2}. 

And the algorithm satisfies \propertyref{3}, in fact, \eqref{euler property 3} applies, for the same reason that the shifts are exactly aligned with the gradient directions.

Because it satisfies all three properties, by \thmref{prototype}, 
\begin{center}
{\em Line Search Shift is consistent}.
\end{center}

\subsection{Mean Shift}
\label{sec:mean shift}

Given a kernel with bandwidth $h > 0$, denoted $K_h(\cdot):=K(\frac{\cdot}{h})$, define the mean shift at a point $x$ as follows
\begin{align} \label{ms_h}
\ms_h(x) := \frac{\int K_h(y - x) (y - x) f(y) \d y}{\int K_h(y - x) f(y) \d y}.
\end{align}
The {\em Mean Shift} algorithm is then defined as constructing the following sequence when initialized at some point $x$:
\begin{equation} \label{mean shift}
x(0) = x; \quad
x(k+1) = x(k) + \ms_h(x(k)), \quad k \ge 0.
\end{equation}

We establish its consistency following, again, the blueprint detailed for the prototype algorithm of \secref{prototype}.
Some elements of consistency for Mean Shift appear in the work of \citet{cheng1995} and \citet{comaniciu2002mean, comaniciu1999mean}, and a few others reviewed in \cite{carreira2015clustering}.
As far as we are aware of, the consistency of Mean Shift per se is established here for the first time.

We start with the following result, which is a continuous version of \cite[Th 1]{cheng1995}.
\begin{lem} \label{lem:ms_h}
Suppose that $K(x) = \sck(\|x\|^2)$ for some nonnegative, nondecreasing, integrable function $\sck$. Define $\scl(u) := c \int_u^\infty \sck(v) \d v$ and $L(x) := \scl(\|x\|^2)$. 
Then
\begin{equation} \label{shadow}
\ms_h(x)
= \frac{h^2}{2c} \frac{\nabla (L_h * f)(x)}{K_h * f(x)}.
\end{equation}
\end{lem}  

\begin{proof}
Assume without loss of generality that $c=1$.
Notice that
\begin{align*}
 \nabla L_h (x) = \nabla \Big[ h^{-d} \int_{\|x/h\|^2}^\infty \sck(v)\d v \Big] = -\frac{2}{h^2} h^{-d} \sck(\|x/h\|^2)x = -\frac{2}{h^2} K_h(x)x.
\end{align*}
Using this, we obtain
\begin{align*}
\nabla (L_h * f)(x) = \nabla L_h * f (x) = \frac2{h^2} \int K_h(y - x) (y - x) f(y) \d y.
\end{align*}
Up to the scaling factor in front, we recognize the numerator in \eqref{ms_h}. And the denominator is simply $K_h * f(x)$.
\end{proof}

If $\sck$ decays fast enough at infinity, $L$ via $\scl$ properly normalized is also a kernel, just like $K$. We assume this is so henceforth. \citet{cheng1995} calls $L$ a `shadow kernel' of $K$. Assuming this is the case, we know from the extensive literature on kernel density estimation that, under some standard regularity conditions on $f$, $K_h * f(x) \to f(x)$ and $\nabla (L_h * f)(x) = L_h * \nabla f(x) \to \nabla f(x)$ as $h \to 0$, implying that 
\begin{equation} \label{F_h}
\frac{\nabla (L_h * f)(x)}{K_h * f(x)} \to \frac{\nabla f(x)}{f(x)}.
\end{equation}
On the right-hand side we recognize the function driving the dynamic system \eqref{original}, indicating that Mean Shift resembles this other algorithm with $\rho = h^2/2c$.
This is true, but to establish consistency, we take a different route which is arguably closer to the blueprint of \secref{prototype}, and overall more conceptual than calculatory.

In view of \eqref{shadow}, Mean Shift can be described by the following Euler scheme
\begin{equation} \label{mean shift euler}
x(0) = x; \quad
x(k+1) = x(k) + \rho_h \frac{\nabla f^\scl_h(x(k))}{f^\sck_h(x(k))}, \quad k \ge 0,
\end{equation}
where $f^\scl_h := L_h * f$ and $f^\sck_h := K_h * f$, and $\rho_h := h^2/2c$.
Note that this is a gradient flow of $f^\scl_h$ --- and not of $f$ --- with varying stepwise that is inversely proportional to $f^\sck_h$.
The idea is to work with that, showing via the blueprint that the algorithm converges to a mode of $f^\scl_h$, and then argue that such a mode is close to a mode of $f$ when the bandwidth $h$ is small.

Before we start, we note two things make this different from \eqref{original variant} with $f^\scl_h$ in place of $f$. One is that the step size is not a function of $f^\scl_h$ but of $f^\sck_h$. The other is that $f^\scl_h$ depends on the step size, which is $\rho_h$. 
But it turns out that we can work with this.

By our smoothness assumption on the density $f$, we have that $f^\scl_h$ is twice differentiable with $\nabla f^\scl_h = L_h * \nabla f$ and $\nabla^2 f^\scl_h = L_h * \nabla^2 f$. In particular, we have $\sup_x\|\nabla f^\scl_h(x)\| \le \sup_x\|\nabla f(x)\|$ and $\sup_x\|\nabla^2 f^\scl_h(x)\| \le \sup_x\|\nabla^2 f(x)\|$, which then implies that \eqref{kappa1} and \eqref{kappa2}, and thus also \eqref{taylor2}, apply to $f^\scl_h$ with the meaning of $\kappa_1$ and $\kappa_2$ unchanged. In addition, we have the zeroth, first, and second derivatives of $f^\sck_h$ converge uniformly to those of $f$ as $h \to 0$. Let 
\begin{align}\label{eta_all}
\eta^{\scl, 0}_h := \sup_x |f^\scl_h(x) - f(x)|, \quad
\eta^{\scl, 1}_h := \sup_x \|\nabla f^\scl_h(x) - \nabla f(x)\|, \quad
\eta^{\scl, 2}_h := \sup_x \|\nabla^2 f^\scl_h(x) - \nabla^2 f(x)\|,
\end{align}
so that 
\begin{align}\label{eta}
\eta^\scl_h := \max\{\eta^{\scl, 0}_h, \eta^{\scl, 1}_h, \eta^{\scl, 2}_h\} \to 0, \quad \text{as $h \to 0$.}
\end{align}
The same is true of $f^\sck_h$, meaning that 
\begin{align}
\eta^\sck_h := \max\{\eta^{\sck, 0}_h, \eta^{\sck, 1}_h, \eta^{\sck, 2}_h\} \to 0, \quad \text{as $h \to 0$,}
\end{align} 
with analogous definitions, although we will only use the zeroth order convergence.
In fact, more is true.
\begin{lemma}
\label{lem:stability}
For any non-critical level $t$ of $f$, as $h \to 0$, within $\up_t$, the modes of $f^\scl_h$ and their basins of attraction, as well as its $t$-upper level set, converge to those of $f$.
\end{lemma}

\begin{proof}
The result immediately follows from \lemref{mode_consistency} and \lemref{basin consistency} given in \secref{methods}.
\end{proof}


\begin{lemma}
\label{lem:mean shift cluster}
Take a non-critical level $t >0$. For $h$ small enough, the following happens. Let $\cC$ be a level cluster for $f^\scl_h$. Then, when initialized inside $\cC$, Mean Shift is hill-climbing  for $f^\scl_h$ and converges to a critical point of $f^\scl_h$ inside $\cC$. In particular, Mean Shift satisfies \propertyref{2} with respect to $f^\scl_h$.
\end{lemma}

\begin{proof}
Assume that $h$ is small enough that $\eta^\scl_h \le t/2$, so that for any $x$ in any $t$-level cluster for $f^\scl_h$, $f(x) \ge f^\scl_h(x) - \eta^\scl_h \ge t - t/2 = t/2$.
Take $h$ even smaller that $2 \eta^\scl_h + \eta^\sck_h \le t/2$ and that $\kappa_2 \rho_h = \kappa_2 h^2/2 \le t/2$.
And still smaller that all the $t$-level clusters are separated by more than $\rho_h (\kappa_1 + \eta^\scl_h)/(t/2)$. The latter is possible because as $h \to 0$ the $t$-level clusters of $f^\scl_h$ converge to those of $f$ by \lemref{stability}, and the latter are disjoint.

Now that $h$ is fixed, we proceed. Let $\cC$ denote a $t$-level cluster for $f^\scl_h$ and initialize the process at some $x_0 \in \cC$, and let $(x_k)$ denote the resulting sequence. 
Using a second-order Taylor development for $f^\scl_h$, we have
\begin{align}
f^\scl_h(x_{k+1}) - f^\scl_h(x_{k})
&= \rho_h \frac{\|\nabla f^\scl_h(x_k)\|^2}{f^\sck_h(x_k)} \pm \frac{\kappa_2}2 \rho_h^2 \frac{\|\nabla f^\scl_h(x_k)\|^2}{f^\sck_h(x_k)^2} \\
&\ge \tfrac{1}{2} \rho_h \frac{\|\nabla f^\scl_h(x_k)\|^2}{f^\sck_h(x_k)} \quad \text{when} \quad \kappa_2 \rho_h \le f^\sck_h(x_k). \label{mean shift proof1}
\end{align}

We prove by induction that $(f^\scl_h(x_k))$ is strictly increasing until convergence if it converges in a finite number of steps.
Suppose that we have shown that $f^\scl_h(x_0) \le \cdots \le f^\scl_h(x_k)$, which is certainly true at $k = 0$.
Because 
\begin{align}
\kappa_2 \rho_h 
\le t/2
\le f(x_0)/2
&\le f(x_0) - 2 \eta^\scl_h - \eta^\sck_h \\
&\le f^\scl_h(x_0) - \eta^\scl_h - \eta^\sck_h \\
&\le f^\scl_h(x_k) - \eta^\scl_h - \eta^\sck_h \\
&\le f(x_k) - \eta^\sck_h \\
&\le f^\sck_h(x_k),
\end{align}
the condition in \eqref{mean shift proof1} is satisfied, and so the inequality holds, showing that the induction carries on.
Therefore, the algorithm is hill-climbing for $f^\scl_h$.

In the process, we have shown that $f^\scl_h(x_k) \ge t/2$ for all $k$, implying the following upper bound on the shift size
\begin{align}
\|x_{k+1} - x_k\|
= \rho_h \frac{\|\nabla f^\scl_h(x_k)\|}{f^\sck_h(x_k)}
\le \rho_h \frac{\kappa_1 + \eta^\scl_h}{t/2}.
\end{align}
By assumption on $h$, the upper bound is strictly smaller than the separation between the $t$-level clusters of $f^\scl_h$, implying that the sequence remains in $(x_k)$ because it is hill-climbing for $f^\scl_h$.

The remaining arguments underlying \lemref{euler cluster} carry over verbatim to establish the entire statement, although here with reference to $f^\scl_h$ instead of $f$.
\end{proof}

Just like any other Euler scheme, \propertyref{3} with respect to $f^\scl_h$ is trivially satisfied since
\begin{align}
x_{k+1} - x_k
= \|x_{k+1} - x_k\| \|\nabla f^\scl_h(x_k)\|.
\end{align}

As for \propertyref{1}, it is also satisfied since, in the same context of \lemref{mean shift cluster}, we have
\begin{align}
\rho_h \frac{\|\nabla f^\scl_h(x_k)\|}{\kappa_0} 
\le \|x_{k+1} - x_k\|
= \rho_h \frac{\|\nabla f^\scl_h(x_k)\|}{f^\sck_h(x_k)}
\le \rho_h \frac{\|\nabla f^\scl_h(x_k)\|}{t/2},
\end{align}
using the fact that $\sup_x f^\sck_h(x) \le \sup_x f(x)$ and $f^\sck_h(x_k) \ge t/2$ shown above in the induction.

Because it satisfies all three properties, by \thmref{prototype}, we may conclude (with some care here) that, given a level $s > 0$, for $h$ small enough, when initialized at any point in 
$f^\scl_h(x_0) \ge s$ that is in the basin of attraction of some mode of $f^\scl_h$, Mean Shift converges to that mode. 

This is not quite what we want, as we are interested in the modes of $f$ and their basins of attraction. Here we invoke \lemref{stability}. Take any point $x_0$ such that $f(x_0) > 0$ and take any level $s > 0$ such that $f(x_0) \ge 2 s$. Suppose in addition that $x_0$ is in the basin of attraction $\cB$ of some mode $x_*$ of $f$.
By the stability lemma, to each $h$ small enough we may associate a mode $x^\scl_{h, *}$ of $f^\scl_h$ with basin of attraction $\cB^\scl_h$ such that $x^\scl_{h, *} \to x_*$ and $\cB^\scl_h \to \cB$ as $h \to 0$.
For $h$ small enough, we have $x_0 \in \cB^\scl_h$.  
For $h$ even smaller, $f^\scl_h(x_0) \ge s$.
And for $h$ still even smaller, the previous statement holds.
When $h$ is that small, Mean Shift initialized at $x_0$ converges to $x^\scl_{h, *}$. And because $x^\scl_{h, *} \to x_*$ as $h \to 0$, we may conclude that 
\begin{center}
{\em Mean Shift is consistent}.
\end{center}

\section{Consistency: Methods}
\label{sec:methods}

As announced in the Introduction, we call `method' an algorithm applied to an estimate of the density computed based on a sample of points. As is usually the case, we assume that the sample was generated iid from the underlying density we are ultimately interested in. In this context, we establish the large-sample ($n\to\infty$) consistency of the methods associated with the various algorithms considered in \secref{algorithms}.

Continuing with same definition of clustering, we say that a method is {\em consistent} if it moves a fraction tending to one of data points to their associated mode when the sample size is sufficiently large --- and the bandwidth defining the density estimate is appropriately chosen. 

\subsection{Some stability results}

We first present some important results regarding the stability of the upper level sets, modes and their basins of attraction for density functions. Consider a density function $\tilde f$ on $\bbR^d$, which is to be understood as a perturbed version of $f$, and is really a placeholder for a kernel density estimator of $f$ in practice. Denote $\tilde \up_s := \{x: \tilde f(x) \ge s\}$, which is the upper $s$-level set of $\tilde f$. Let 
\begin{align}
\eta_{0} := \sup_x |\tilde f(x) - f(x)|, \quad
\eta_{1} := \sup_x \|\nabla \tilde f(x) - \nabla f(x)\|, \quad
\eta_{2} := \sup_x \|\nabla^2 \tilde f(x) - \nabla^2 f(x)\|.
\end{align}
We use $\cF_\eta$ to denote the class of density functions $\tilde f$ satisfying $ \max\{\eta_0,\eta_1,\eta_2\} \le \eta$, and the same second condition as $f$ listed in \secref{density}. For an arbitrary quantity $g$ that tends to zero with $\eta$, we write $g=\omega(\eta)$.

\begin{lemma}
\label{lem:mode_consistency}
In the present context, the following is true: \\
(a) For any $s\in(0,\kappa_0)$ such that $\cL_s$ does not contain a critical point of $f$: \\
\indent (a1) $\sup\{\haus(\up_s,\tilde \up_s) : \tilde f \in \cF_\eta\} = \omega(\eta)$;\\
\indent (a2) for any $\tilde f \in \cF_\eta$, all the critical points of $\tilde f$ in $\up_s$ are non-degenerate, when $\eta$ is small enough. \\
(b) For any mode $x_*$ of $f$, there is $\delta_*>0$ such that, when $\eta$ is small enough, any function $\tilde f \in \cF_\eta$ has only one critical point (in fact, a mode) in $\ball(x_*,\delta_*)$; moreover, if we denote the mode of $\tilde f$ in $\ball(x_*,\delta_*)$ by $x_*[\tilde f]$, we have $\sup\{\|x_* - x_*[\tilde f]\| : \tilde f \in \cF_\eta\} =\omega(\eta)$. \\
(c) Let $\cM$ denote the set of all the modes of $f$ and $\cM[\tilde f]$ denote the same for an arbitrary function $\tilde f$. For any non-critical level $s$ of $f$, we have $\sup\{\haus(\up_s \cap \cM, \up_s \cap \cM[\tilde f]) : \tilde f \in \cF_\eta\} = \omega(\eta)$.
\end{lemma}

We note that similar `stability' results can be found in the literature, e.g., \cite[Th 1]{cuevas2006plug}, \cite[Lem 8]{arias2016estimation}, \cite[Lem 3]{genovese2016non}.

\begin{proof}
(a1) Notice that for any $\tilde f \in \cF_\eta$ and for $\eta_0$ small enough that $s+2\eta_0<\kappa_0$ and $s-2\eta_0>0$, 
\begin{align}
\up_{s+2\eta_0} \subset \tilde \up_{s+\eta_0} \subset \up_s \subset \tilde \up_{s-\eta_0} \subset \up_{s-2\eta_0} .
\end{align}
Therefore, using the fact that the upper level sets of any function are monotone for the inclusion,
\begin{align}
\label{Haus_upper}
\haus(\up_s,\tilde \up_s) & = \max\big\{\haus(\up_s \mid \tilde \up_s), \haus(\tilde \up_s \mid \up_s)\big\} \nonumber\\
& \le \max\big\{\haus(\up_s \mid \up_{s+2\eta_0}), \haus(\up_{s-2\eta_0} \mid \up_s)\big\} \nonumber\\
&\le \haus(\up_{s-2\eta_0},\up_{s+2\eta_0}) \nonumber\\
& = \haus(\cL_{s-2\eta_0},\cL_{s+2\eta_0}),
\end{align}
where the last equality holds because of the monotonicity of the upper level sets as stated above and the fact that $\cL_v$ is the boundary of $\up_v$ for all $v\in (0,\kappa_0)$. 
Suppose that $\eta_0$ is small enough that there exists no critical points of $f$ at any levels anywhere between $s-2\eta_0$ and $s+2\eta_0$, inclusive. This is possible because $\cL_{s}$ does not contain critical points, and the number of critical values of $f$ above, say, $s/2$, is finite.
The upper bound in \eqref{Haus_upper} converges to zero as $\eta_0 \to0 $, as we have already shown in recent, related work \cite[Th 4.2]{arias2021level}.

(a2) The result follows immediately from \cite[Lem 5.32]{banyaga2013lectures}.

(b) Let $-\lambda_*<0$ be the largest eigenvalue of $\nabla^2 f(x_*)$. Due to the continuity of the second derivatives of $f$, there exists $\delta_*>0$ such that the largest eigenvalue of $\nabla^2 f(x)$ is upper bounded by $-\frac{1}{2}\lambda_*$ for all $x\in \bar\ball(x_*,\delta_*)$. A Taylor expansion gives, for all such $x$,
\begin{align}
f(x_*) - f(x) \ge \frac{1}{4} \lambda_* \|x_*-x\|^2.
\end{align}
Suppose that $\eta_0$ is small enough that $\eta_0 < \frac{1}{8} \lambda_* \delta_*^2$. For any $x\in \partial\ball(x_*,\delta_*)$ and any $\tilde f \in \cF_\eta$,
\begin{align}
\tilde f(x_*) \ge f(x_*) -\eta_0 \ge f(x) + \frac{1}{4} \lambda_* \delta_*^2 - \eta_0 \ge \tilde f(x) + \frac{1}{4} \lambda_* \delta_*^2 - 2\eta_0 >\tilde f(x).
\end{align}
Hence there must exist $k\ge1$ modes of $\tilde f$ in $\ball(x_*,\delta_*)$. Below we show that $k=1$ when $\eta_2$ is small enough. Because of the 1-Lipschitz continuity of the largest eigenvalue as a function on the space of symmetric matrices, the largest eigenvalue of $\nabla^2 \tilde f(x)$ is upper bounded by $-\frac{1}{4}\lambda_*$ for all $x\in\bar\ball(x_*,\delta_*)$, when $\eta_2 \le \frac{1}{4}\lambda_*$. 
Note that this implies that any critical point of $\tilde f$ in $\bar\ball(x_*,\delta_*)$ must be one of its modes. 
Let $\tilde x_*$ be a mode of $\tilde f$ in $\ball(x_*,\delta_*)$. Then for any $x\in \bar\ball(x_*,\delta_*)$, it follows from a Taylor expansion that 
\begin{align}
\label{f_tilde_diff}
\tilde f(x) - \tilde f(\tilde x_*) \le -\frac{1}{8}\lambda_* \|x-\tilde x_*\|^2.
\end{align}
In other words, $\tilde f(x) < \tilde f(\tilde x_*)$, for all $x\in\bar\ball(x_*,\delta_*)$ different from $\tilde x_*$, which implies that $\tilde x_*$ is the only mode of $\tilde f$ in $\ball(x_*,\delta_*)$. Letting $x=x_*$ in \eqref{f_tilde_diff}, we obtain
\begin{align}
\label{f_tilde_mode}
\tilde f(x_*) - \tilde f(\tilde x_*) \le -\frac{1}{8}\lambda_* \|x_*-\tilde x_*\|^2.
\end{align}
And using a similar Taylor expansion for $f$ derived above,
\begin{align}
f(\tilde x_*) - f(x_*) \le -\frac{1}{4}\lambda_* \|x_*-\tilde x_*\|^2.
\end{align}
Combining the above two inequalities yields
\begin{align}
\label{mode_distance}
\|x_*-\tilde x_*\|^2 \le \frac{8}{3\lambda_*} \{[f(x_*) - \tilde f(x_*)] + [\tilde f(\tilde x_*) - f(\tilde x_*)]\} \le \frac{16}{3\lambda_*} \eta_0 = :C_0 \eta_0.
\end{align}
Therefore $\sup_{\tilde f \in \cF_\eta}\|x_*-\tilde x_*\| \to 0$ as $\eta \to 0.$

(c) 
Suppose there exists a critical point $x_{\dagger}$ of $f$ in $\cU_s$, which is not a mode. Then the largest eigenvalue of $\nabla^2 f(x_{\dagger})$ must be positive, denoted by $\lambda_\dagger$. Using the continuity of the second derivatives of $f$, there exists $\delta_\dagger>0$ such that the largest eigenvalue of $\nabla^2 f(x)$ is lower bounded by $\frac{1}{2}\lambda_\dagger$ for all $x\in \bar\ball(x_\dagger,\delta_\dagger)$. Suppose that $\eta_2\le \frac{1}{4}\lambda_\dagger$. Then the largest eigenvalue of $\nabla^2 \tilde f(x)$ is lower bounded by $\frac{1}{4}\lambda_\dagger$ for all $x\in \bar\ball(x_\dagger,\delta_\dagger)$, implying that there is no mode of $\tilde f$ in $\bar\ball(x_\dagger,\delta_\dagger)$. 

Since $\cU_s$ is a compact set, there are only finitely many critical points of $f$ in $\cU_s$. Denote the union of all the balls near the critical points constructed above, including those in (b), by $\cB$. Then $\cU_s\setminus \cB$ is a compact set, and there is no critical point of $f$ in $\cU_s\setminus \cB$ by construction. Using the continuity of $\|\nabla f\|$, there is a positive lower bound of $\|\nabla f\|$ on $\cU_s\setminus \cB$, denoted by $c$. Now suppose that $\eta_1\le \frac{1}{2}c$. Then $\|\nabla \tilde f(x)\| \ge \frac{1}{2} c$ for all $x\in \cU_s\setminus \cB$, and hence there is no critical point of $\tilde f$ in $\cU_s\setminus \cB$. Therefore, our construction guarantees that there is no mode of $\tilde f$ on $\cU_t$, except that there is only one mode $\tilde x_*$ of $\tilde f$ in $\ball(x_*,\delta_*)$ for each mode $x_*$ of $f$ with its corresponding radius $\delta_*$ as defined above. Using \eqref{mode_distance} we conclude that $f$ and $\tilde f$ have the same number of modes on $\cU_t$, and $\sup_{\tilde f \in \cF_\eta}\haus(\up_s \cap \cM, \up_s \cap \cM[\tilde f])\to 0$ as $\eta \to 0$.
\end{proof}

\begin{lemma}
\label{lem:basin consistency}
In the present context, the following is true: \\
(a) Any point $x_0$ in the basin of attraction of $x_*$ relative to $f$ is also in the basin of attraction of $x_*[\tilde f]$ (as given in \lemref{mode_consistency}) relative to $\tilde f$, for any $\tilde f \in \cF_\eta$,x when $\eta$ is small enough. \\
(b) There is a measurable set $\Theta_\eta$ with probability at least $1-\omega(\eta)$ such that $\sup\{\|\gamma_x(\infty) - \gamma[\tilde f]_x(\infty)\| : x\in\Theta_\eta, \tilde f \in \cF_\eta\} =\omega(\eta)$, where $\gamma_x$ and $\gamma[\tilde f]_x$ are the gradient flows of $f$ and $\tilde f$ starting at $x$, respectively.
\end{lemma}

\begin{proof}
(a) Denote $\tilde N(x) = \nabla \tilde f(x)/\|\nabla \tilde f(x)\|$. Consider the flow of $\tilde N$:
\begin{equation} \label{zeta_tilde}
\tilde\zeta_x(0) = x; \quad
\dot{\tilde\zeta}_x(t) 
= \tilde N(\tilde\zeta_x(t)).
\end{equation}
We want to show that for any $x_0$ in the basin of attraction of $x_*$, $\tilde\zeta:=
\tilde\zeta_{x_0}$ ends at $\tilde x_*$, when $\eta$ is small enough.
(Here and below, we use $\tilde x_*$ in place of $x_*[\tilde f]$ when it is clear what $\tilde f$ is.)

Using \lemref{cluster ball}, there exists a constant $C>0$ such that with $s_*:=f(x_*)$ and $q_* :=s_* - (\delta_*/C)^2>0$, $\cC_*:=\cC_{q_*}(x_*)$ is a leaf cluster of $f$ contained in $\ball(x_*,\delta_*)$. Let $\tilde \cC_s$ be the cluster of $\tilde f$ at level $s$ containing $\tilde x_*$ whenever this is well defined. Denote $\eta_*:=\frac{1}{4}(s_* - q_*)$, and define $s_\ddagger = q_* + \eta_*$, $s_\dagger = q_* + 2\eta_*$, and $s_\triangle=q_* + 3\eta_*$. 
Correspondingly, denote $\cC_\ddagger=\cC_{s_\ddagger}(x_*)$, $\cC_\dagger=\cC_{s_\dagger}(x_*)$, and $\cC_\triangle=\cC_{s_\triangle}(x_*)$. 
Suppose that $\eta_0$ is small enough that \eqref{f_tilde_mode} holds and $\eta_0<\eta_*$, so that 
\begin{align}
\tilde s_* := \tilde f(\tilde x_*) \ge \tilde f(x_*) \ge f(x_*) - \eta_0 > s_\triangle > s_\ddagger,
\end{align}
which means $\tilde \cC_{\ddagger}:=\tilde \cC_{s_\ddagger}(\tilde x_*)$ is well-defined. Using \lemref{cluster ball} again, $\ball(x_*,\delta_\diamond) \subset \cC_*$, where $\delta_\diamond:=\sqrt{s_*-q_*}/C=\delta_*/C^2$. We require $\eta_0 \le \delta_\diamond^2/C_0$ so that $\tilde x_*\in \ball(x_*,\delta_\diamond)$ by \eqref{mode_distance}, implying in turn that $\tilde x_* \in \cC_*$. 
Then we must have $\tilde \cC_{\ddagger} \subset \cC_*$ for the following reason. Take $y \in \tilde \cC_{\ddagger}$. By the fact that $\tilde \cC_{\ddagger}$ is connected, there is a path $p \subset\tilde \cC_{\ddagger}$ connecting $\tilde x_*$ and $y$. Any point $x\in p$ satisfies $f(x) \ge \tilde f(x) - \eta_0 \ge s_\ddag - \eta_0 \ge q_*$, so that $p \subset \up_{q_*}$. By the fact that $p$ intersects $\cC_*$, which is a connected component of $\up_{q_*}$, it must be that $p \subset \cC_*$, and in particular, $y \in \cC_*$.
Hence $\tilde \cC_{\ddagger} \subset \ball(x_*,\delta_*)$ by construction, which also means that $\tilde \cC_{\ddagger}$ is a leaf cluster of $\tilde x_*$ because $\tilde x_*$ is the only critical point (in fact, mode) in $\ball(x_*,\delta_*)$, that is, the gradient flow $\tilde\zeta$ converges to $\tilde x_*$ if initialized at any point $x\in\tilde \cC_{\ddagger}$. By \lemref{cluster ball}, when $\eta_2$ is small enough, we can find $\tilde C\le2C$ such that $\ball(\tilde x_*,\tilde\delta)\subset \cC_\ddagger$, where
\[\tilde\delta := \frac1{\tilde C} \sqrt{\tilde s_* - s_\ddagger} > \frac1{\tilde C} \sqrt{s_\triangle-s_\ddagger} = \frac1{\tilde C} \sqrt{\tfrac{1}{2} (\delta_*/C)^2} > \frac{\delta_*}{4 C^2}.\]
 We require $\eta_0 \le \tilde\delta^2/C_0$ so that $x_*\in \ball(\tilde x_*,\tilde\delta)\subset  \tilde\cC_\ddagger$, by \eqref{mode_distance}. Furthermore, we must have $\cC_\dagger \subset \tilde \cC_{\ddagger}$ for a similar reason as we have argued above for $\tilde \cC_{\ddagger} \subset \cC_*$, because here both $\cC_\dagger$ and $\tilde \cC_{\ddagger}$ are compact and connected, and for all $x\in\cC_{\dagger}$, $\tilde f(x) \ge f(x) - \eta_0 \ge s_\dagger - \eta_0 \ge s_\ddagger$.
Below we show that $\tilde\zeta$ enters $\cC_\dagger$ at some time point, and because $\cC_\dagger\subset\tilde \cC_{\ddagger}$, it must be the case that it ends at $\tilde x_*$. (As everywhere in this proof, this is understood to be true when $\eta$ is small enough.)

Noticing that $\ball(x_*,\delta_\diamond) \subset \cC_*$ and using \lemref{cluster ball}, we have $\bar\ball(x_*,\delta_\diamond/\sqrt{2})\subset\cC_\dagger$. 
Let $t_\# :=\inf\{t\ge0: \|\zeta(t) - x_*\| \le \delta_\diamond/\sqrt{2}\}$, 
and $t_\diamond :=\inf\{t\ge0: \|\zeta(t) - x_*\| \le \delta_\diamond/(2\sqrt{2})\}$. We reuse the same notation as in the proofs of \thmref{prototype} and \lemref{nu_continuity}: $z_\#$, $\nu$, $\cT$, $\delta_{\rm tube}$, $\mu$, $\cS$, $\cV$, and $\delta_{\rm tube}^\prime$. 
%
As shown in the proof of \lemref{nu_continuity}, $N$ is $\kappa$-Lipschitz on $\cV$, where $\kappa = \max\{\kappa_2/\nu + \kappa_2, 8/\delta_{\rm tube}^\prime\}$. Then the following is an immediate result of \cite[Sec 17.5]{hirsch2012differential}:
\begin{equation}
\|\zeta(t) - \tilde\zeta(t)\| \le \frac{\eta_1}{\kappa} [\exp(\kappa t) - 1]=:\psi(t), \quad\forall t\in[0,t_\#].
\end{equation}
Here we require $\eta_1$ to be small enough that $\psi(t_\#) \le \frac{1}{2}\delta_{\rm tube}^\prime$, so that $\tilde \cZ_{t_\#} := \{\tilde\zeta(t): t\in [0,t_\#]\}\subset \cV$.
Let $\tilde z_\#=\tilde\zeta(t_\#)$. Then $\|z_\# - \tilde z_\#\| \le \psi(t_\#)$. We further require $\eta_1$ to be small enough that $\psi(t_\#) \le \delta_\#: =\eta_*/\kappa_1$. For any $y\in\ball(\cC_\triangle,\delta_\#)$, there exists $x\in\cC_\triangle$ such that $\|x-y\|\le \delta_\#$, and hence by \eqref{kappa1},
\begin{align}
f(y) \ge f(x) - \kappa_1\delta_\# \ge s_\triangle - \eta_* = s_\dagger,
\end{align}
which implies $\ball(\cC_\triangle,\delta_\#)\subset \cC_\dagger$. Since $z_\#\in \cC_\triangle$ and $\|z_\# - \tilde z_\#\| \le \delta_\#$, we must have $\tilde z_\#\in\cC_\dagger$.
We may thus conclude that $\tilde\zeta$ ends at $\tilde x_*$.

(b) 
For any $\epsilon>0$, let $s_\epsilon$ be a non-critical level of $f$ such that the probability measure of $\cU_{s_\eps}^{\complement}$ is not larger than $\epsilon$. 
By \lemref{mode_consistency}, with $\eta$ small enough for any $\tilde f\in\cF_\eta$, all the critical points of $\tilde f$ in $\cU_{s_\epsilon}$ are non-degenerate, and $\tilde f$ has the same number of modes as $f$ in $\cU_{s_\epsilon}$, whose locations match in such a way that their Hausdorff distance is $\omega(\eta)$. Following the same arguments as in the proof of \thmref{prototype_uniform}, we can extend the result in (a) to its uniform form, that is, $\gamma_x(\infty)$ and $\tilde\gamma_x(\infty)$ match in the above sense, for any starting point $x$ in $\cU_{s_\epsilon}$, except for a small tube around the boundary of the basins of attraction of all the modes of $f$ in $\cU_{s_\epsilon}$, whose probability measure tends to zero as $\eta\to 0$.  Since $\epsilon$ is arbitrarily small, we arrive at the conclusion of this theorem.
\end{proof}

\subsection{Consistency of Hill-Climbing Methods}
\label{sec:max (slope) shift method}
In this subsection, we show the consistency of various hill-climbing methods whose algorithms have been discussed in \secref{algorithms}. When an i.i.d.~sample $\{x_1,\dots,x_n\}$ from the density $f$ is available, we estimate the density by the kernel density estimator (KDE)
\begin{align}
\hat f_{n,h}(x) = \frac{1}{nh^d} \sum_{i=1}^n K_h( x - x_i).
\end{align}
For simplicity, we take $K(x) = \sck(\|x\|^2)$ with a shadow kernel $L(x)=\scl(\|x\|^2)$, and use $\hat f_{n,h}^\scl$ to denote the KDE with the kernel $L$. 
While the algorithms use the knowledge of the density and its derivatives, the methods rely on the KDE and its derivatives. 
The sequences $\{\hat x(k): k=0,1,\cdots,\}$ generated by these methods are defined by replacing $f$ with $\hat f_{n,h}$ (or $\hat f_{n,h}^\scl$ for the Mean Shift) in the algorithms given in \secref{algorithms}. Note that the update in the Mean Shift method can be equivalently written as 
\begin{equation} \label{mean shift method equal}
\hat x(0) = x; \quad
\hat x(k+1) = \frac{\sum_j K_h(x_j - \hat x(k)) x_j}{\sum_j K_h(x_j - \hat x(k))}, \quad k \ge 0.
\end{equation}

Let 
\begin{align}\label{eta_n}
\eta_{n,h}^{\sck,0} := \sup_x |\hat f_{n,h}(x) - f(x)|, \quad
\eta_{n,h}^{\sck,1} := \sup_x \|\nabla \hat f_{n,h}(x) - \nabla f(x)\|, \quad
\eta_{n,h}^{\sck,2} := \sup_x \|\nabla^2 \hat f_{n,h}(x) - \nabla^2 f(x)\|,
\end{align}
and 
\[\eta_{n,h}^\sck := \max\{\eta_{n,h}^{\sck,0},\eta_{n,h}^{\sck,1},\eta_{n,h}^{\sck,2}\}.\] 
Define $\eta_{n,h}^\scl$ analogously by replacing $\sck$ with $\scl$ in the above notation. Denote $\hat \up_t = \{x: \hat f_{n,h}(x) \ge t\}$. We first establish the consistency of all the above methods except for Mean Shift.

We assume that $K: \bbR^d\to\bbR$ is nonnegative, twice continuously differentiable, and compactly supported. Note that $L$ has the same property. The following is a standard uniform consistency result for KDE and its derivatives. See, for example, \cite[Lem 2, Lem 3]{arias2016estimation}.
\begin{lemma}
\label{lem:kde consistency}
Suppose that $h = h_n$ is chosen such that $h\to 0$ and $\log n/(nh^{d+4}) \to 0$ as $n\to \infty$. 
Then for any $\eta>0$, there exists $n_0>0$ such that for all $n\ge n_0$, with probability at least $1-\eta$,
\begin{align}
\max(\eta_{n,h}^\scl, \eta_{n,h}^\sck) \le \eta. 
\end{align}
\end{lemma}

We are ready to give the consistency results of the methods considered in this paper. The proofs are straightforward using the results that have been established in this and previous sections, and hence are omitted.

We first consider all the hill-climbing methods except the Mean Shift. We use $\theta$ to denote the neighborhood size parameter, meaning $\eps$ or $\rho$ for the corresponding methods. Using \thmref{prototype}, \lemref{mode_consistency}, \lemref{basin consistency}, and \lemref{kde consistency}, we immediately have the following result.
\begin{theorem}
Suppose the conditions of \lemref{kde consistency} are satisfied.
Consider a mode $x_*$ of $f$. For any point $x_0$ in the basin of attraction of $x_*$, the following is true with probability tending to one as $n\to\infty$: when initialized at $x_0$, for $\theta$ small enough, each of the above hill-climbing methods excluding the Mean Shift produces a sequence that converges to $x_*$.
\end{theorem}

The Mean Shift method differs from all the other methods in that its sequence converges to a mode of $\hat f_{n,h}^\scl$ rather than $\hat f_{n,h}$, and its neighborhood size parameter $h$ is the bandwidth used in KDE. The following theorem gives the consistency of the Mean Shift method. The proof follows the same arguments as in \secref{mean shift}, where $\hat f_{n,h}$ and $\hat f_{n,h}^\scl$ play the role of $f_h^\sck$ and $f_h^\scl$.

\begin{theorem}
Suppose the conditions of \lemref{kde consistency} are satisfied. Consider a mode $x_*$ of $f$. For any point $x_0$ in the basin of attraction of $x_*$, the following is true with probability tending to one as $n\to\infty$: when initialized at $x_0$, the Mean Shift method produces a sequence that converges to $x_*$.
\end{theorem}

We also have uniform consistency results for all the hill-climbing methods. Again, we first consider all the hill-climbing methods except the Mean Shift. The result is a consequence of \thmref{prototype_uniform}.

\begin{theorem}
\label{thm:uniform consistency methods}
Suppose the conditions of \lemref{kde consistency} are satisfied. There is a measurable set $\Xi_{\theta,n,h}$ with probability measure tending to one as $\theta \to 0$ and $n\to\infty$, such that starting from any $x\in\Xi_{\theta,n,h}$, the sequence produced by each of the above hill-climbing methods excluding the Mean Shift converges to $\gamma_x(\infty)$.
\end{theorem}

Next, we establish the uniform consistency of the Mean Shift method. 
\begin{theorem}
\label{thm:uniform consistency MS methods}
Suppose the conditions of \lemref{kde consistency} are satisfied. There is a measurable set $\Xi_{n,h}$ with probability measure tending to one as $n\to\infty$, such that starting from any $x\in\Xi_{n,h}$, the sequence produced by the Mean Shift method converges to $\gamma_x(\infty)$.
\end{theorem}

\section{Medoid variants}
\label{sec:medoids}

While Euler Shift \cite{fukunaga1975} and Mean Shift \cite{fukunaga1975, cheng1995} were defined as continuous-space algorithms, Max Slope Shift \cite{koontz1976graph} was first introduced as a discrete-space algorithm: Assuming that a density $f$ is provided, and that a (locally finite) set of points $\cY$ is available, a point $x_0$ is moved as follows
\begin{equation} \label{max slope shift}
x(0) = x_0; \quad
x(k+1) \in \argmax_{y \in \cY \cap \ball(x(k), \eps)} \frac{f(y) - f(x(k))}{\|y-x(k)\|}, \quad k \ge 0.
\end{equation}
Ties are broken in some prescribed way, and if $f(x(k+1)) = f(x(k))$, the process stops. 
In \cite{koontz1976graph}, based on a sample, the density is estimated using a kernel, and the point set is the sample itself. 
Note that, defined as such, the algorithm is well-defined and, in particular, does not require regularization. 
Max Shift \cite{chazal2013persistence} --- even though well-defined as a continuous-space algorithm as we did in \eqref{max shift density} --- was also first introduced as a discrete-space algorithm: In the same context as \eqref{max slope shift}, 
\begin{equation} \label{max shift}
x(0) = x_0; \quad
x(k+1) \in \argmax_{y \in \cY \cap \ball(x(k), \eps)} f(y), \quad k \ge 0,
\end{equation}
where, again, the density is estimated in some way, and the point set $\cY$ is the sample itself.

We chose to work with the continuous-space versions of Max Shift and Max Slope Shift for simplicity. But it turns out there are good, practical reasons to work with discrete-space versions (beyond the fact that everything is necessarily discrete when implemented on a computer).
For example, \citet*{sheikh2007mode} introduced a discrete-space variant of Mean Shift, which they named {\em Medoid Shift}, motivated by their view that ``the relationship between the {\sf medoidshift} algorithm and {\sf meanshift} algorithm is similar to the relationship between the {\sf k-medoids} and the {\sf k-means} algorithms":
Based on a kernel $\sck$ and a sample $\cY$, a point $x_0$ is moved as follows
\begin{equation} \label{medoid shift}
x(0) = x_0; \quad
x(k+1) \in \argmin_{y \in \cY} \sum_{y_* \in \cY} \|y_*-y\|^2 \sck(\|y_*-y\|^2/h^2), \quad k \ge 0.
\end{equation}
We adopt this terminology, and from now on refer to a discrete-size version of an algorithm as the medoid version, as in ``medoid Max Shift".

Note that for a medoid algorithm, consistency is also (necessarily) as the medoid set becomes dense enough.

\subsection{Advantages}
One of the advantages highlighted in \cite{sheikh2007mode} is that the algorithm can be easily adapted to the metric setting in which computing a mean may not even make sense --- which is also one of the main advantages of k-Medoids over k-Means. Indeed, all it takes is replacing in \eqref{medoid shift} the Euclidean metric with the available metric. The same is true of the medoid versions of Max Shift and Max Slope Shift. 

There is another advantage that is not discussed in \cite{sheikh2007mode}, but is really a basic principle in optimization: use a coarse discretization, effectively trading some accuracy for some computational complexity. Take Max Shift, for example. When implemented on a computer, the continuous-space maximization that happens at each stage in \eqref{max shift density} needs to be discretized. This is often done by exhaustive search on a grid. (Some form of gradient ascent could be used, but the method would then be near identical to Euler Shift.) To avoid the use of a grid, which becomes quickly impractical in higher dimensions (very quickly in fact, as it is already impractical at $d \ge 5$), we can focus on the sample and simply use the medoid version of Max Shift. But if the sample is very large, some subsampling may help speedup the computations, perhaps substantially.    

In the literature on the mean shift algorithm, a closely related idea appears in work of \citet*{jang2021meanshift++}, would propose {\sf MeanShift{\small\bf ++}}, a variant of Blurring Mean Shift where the points are binned using a regular partition of the space and means are taken over adjacent bins. This variant is shown in numerical experiments to be much faster. 
The authors state that {\sf MeanShift{\small\bf ++}} ``runs in $O(n 3^d)$ per iteration vs $O(n^2 d)$ for [Blurring] {\sf MeanShift}", where $n$ denotes the sample size. 

We claim that consistent clustering can be achieved with a computational cost of $O(n d)$ by using a relatively small subsample as medoid set. We focus on Max Shift, as it is the simplest algorithm among those studied in \secref{algorithms}.


\subsection{Medoid Max Shift}

Consider therefore Max Shift in its medoid form \eqref{max shift}. We assume for now that a density $f$, satisfying the usual assumptions listed in \secref{density}, is available. We assume that the medoid set $\cY$ is finite (or at least locally finite).
Although the starting point can be any point, we start by looking at how a medoid point is moved by the algorithm. Note that the first point computed by the algorithm is necessarily a medoid point (assuming there is a medoid point within $\eps$ of the starting point).

We adapt the arguments given to establish the consistency of Max Shift in \secref{max shift}, and our main tool to do so is simply the triangle inequality. 

Define, for a level $s>0$, 
\begin{align}
\label{alpha}
\alpha_s := \sup_{x \in \up_s} \inf_{y \in \cY} \|x-y\|.
\end{align}
Thus $\alpha_s$ quantifies how dense the medoid set is in the upper $s$-level set. Note that $s \mapsto \alpha_s$ is non-increasing.

\begin{lemma}
\label{lem:medoid max shift level cluster}
Take a level $s>0$ and consider any $2s$-level cluster $\cC$. Then medoid Max Shift initialized at any point in $\cC$ converges (in a finite number of steps) to a (medoid) point in $\cC$ where the gradient has norm bounded by $C (\alpha_s/\eps+\eps)$.
\end{lemma}

\begin{proof}
Let $\alpha$ be short for $\alpha_s$. Note that $s$ is fixed, and when we assume that $\alpha$ is small enough, it is simply a condition on $\cY$ being dense enough in $\up_s$.
Because the gradient is globally bounded, it is enough to show the result for $\alpha/\eps$ and $\eps$ small enough. We start by assuming that $\alpha \le \eps/2$ and that $\eps$ is smaller than the minimum separation between $\cC$ and any other $s$-level cluster.
Then, initialized at an arbitrary point $x_0$ in $\cC$, medoid Max Shift  being hill-climbing, necessarily the sequence of medoids it computes must remain in $\cC$. 
Let that sequence be denoted $(y_k)$. 

It is also the case that the sequence of density values $(f(y_k))$ is strictly increasing until convergence, if it is the case that the sequence converges. 
But because there are finitely many medoids in $\cC$, the sequence must converge in finitely many steps, say $K$, and by construction, the endpoint $y_K$ must satisfy $f(y_K) \ge f(y)$ for all $y \in \cY \cap \ball(y_K, \eps)$.  
We claim that
\begin{equation}
\label{medoid max shift convergence proof1}
\|\nabla f(y_K)\|
\le C_1 (\alpha/\eps+\eps).
\end{equation}
Again, we only need to prove this for $\eps$ small enough, and we assume that it satisfies $\eps \le s/\kappa_1$. 
Let $y$ be a medoid closest to $z := y_K + (\eps-\alpha) N(y_K)$. 
By \eqref{kappa1},  
\begin{align}
\label{medoid max shift level cluster proof1}
f(z) 
\ge f(y_K) - \kappa_1 \|z-y_K\| 
\ge 2s - \kappa_1 \eps \ge s,
\end{align}
since $y_K \in \cC$, $\|z-y_K\| \le \eps$, and our assumption on $\eps$. Therefore, $\|y-z\| \le \alpha$ by definition of $\alpha$.
Then, by the triangle inequality, $\|y - y_K\| \le \eps-\alpha + \alpha \le \eps$, so that $f(y_K) \ge f(y)$.
But using \eqref{kappa2}, \eqref{taylor2}, and the definition of $\kappa_1$, in that order, we derive
\begin{align}
0 \ge f(y) - f(y_K)
&\ge f(z) - f(y_K) - \kappa_1 \|y-z\| \\
&\ge \nabla f(y_K)^\top (z-y_K) - \tfrac12 \kappa_2 \|z-y_K\|^2 - \kappa_1 \alpha \\
&\ge (\eps-\alpha) \|\nabla f(y_K)\| - \tfrac12 \kappa_2 \eps^2 - \kappa_1 \alpha \\
&\ge \eps \|\nabla f(y_K)\| - \kappa_1 \alpha - \tfrac12 \kappa_2 \eps^2 - \kappa_1 \alpha,
\end{align}
from which we get that $\|\nabla f(y_K)\| \le 2\kappa_1 \alpha/\eps + \tfrac12 \kappa_2 \eps$, confirming \eqref{medoid max shift convergence proof1}.
\end{proof}

Below, we deviate from \propertyref{1}, but leave to the reader to either establish that property, or simply adapt the arguments in the proof of \thmref{prototype}, which is easily done. 
Indeed, even a cursory look at the proof of that theorem reveals that all that is required is that a nontrivial fraction of the shifts are of size comparable to that of the largest shift, a fact which is then used to bound the total number of shifts from above --- see \eqref{prototype property 1} and its surroundings.
\begin{lemma}
\label{lem:medoid max shift step size}
In two consecutive shifts in any sequence produced by medoid Max Shift, at least one is size between $\eps/2$ and $\eps$, except possibly for the very last shift. 
\end{lemma}

\begin{proof}
Suppose that, for some $k$, $\|y_{k+1} - y_k\| \le \eps/2$. If the sequence does not end at $y_{k+1}$ but continues, necessarily $y_{k+2}$ needs to be outside $\ball(y_k, \eps)$ because $f(y_{k+2}) > f(y_{k+1})$ and $y_{k+1}$ is a maximum among medoids in $\ball(y_k, \eps)$. And this forces $\|y_{k+2} - y_{k+1}\| > \eps/2$, by the triangle inequality.
\end{proof}

\begin{lemma}
\label{lem:medoid max shift mode}
Suppose that $x_*$ is a mode and take any level $s$ such that $f(x_*) \ge 2s$. Then there is $\delta > 0$ such that medoid Max  Shift initialized at any point in $\ball(x_*, \delta)$ converges to a (medoid) point within distance $C (\alpha_s/\eps+\eps)$ of $x_*$.
\end{lemma}

Thus, medoid Max Shift satisfies a weaker version of \propertyref{2}, where the convergence is not to a mode, but to a medoid point not too far from a mode.

\begin{proof}
By our assumptions on $f$ in \secref{density}, there exist $\lambda>0$ and $\delta_1 > 0$ such that for all $x\in\ball(x_*,\delta_1)$, all the eigenvalues of $\nabla^2 f(x)$ are bounded from above by $-\lambda$. 
Using the fact that
\begin{align}
\nabla f(x)
&= \nabla f(x) - \nabla f(x_*) \\
&= \int_0^1 \nabla^2 f(u x + (1-u) x_*) (x-x_*) \d u,
\end{align}
we then deduce, whenever $\|x-x_*\| \le \delta_1$, that
\begin{align}
\label{max shift mode proof1}
\|\nabla f(x)\| \ge \lambda \|x-x_*\|.
\end{align}
Note that this implies that there are no other critical points, and therefore no other modes, inside $\ball(x_*, \delta_1)$.

Let $s_1 = \max\{f(x) : x \in \partial\ball(x_*, \delta_1)\}$ and take $\delta_1$ even smaller if needed to have $s_1 \ge s$, which is possible since, by construction, $s_1 < s_* := f(x_*)$, and $s_1$ approaches $f(x_*)$ as $\delta_1$ approaches~0. Note that $\cC_{s_1}(x_*) \subset \ball(x_*, \delta_1)$; and by \lemref{cluster ball}, there is $\delta_2 \le \delta_1$ such that $\ball(x_*, \delta_2) \subset \cC_{s_1}(x_*)$. 

Now take a starting point $x_0 \in \ball(x_*, \delta_2)$ and let $(y_k)$ denote the sequence produced by the algorithm. By \lemref{medoid max shift level cluster}, based on the fact that $x_0 \in \cC_{s_1}(x_*)$, we assert that $(y_k)$ converges to a medoid $y_\infty$ within $\cC_{s_1}(x_*)$ satisfying $\|\nabla f(y_\infty)\| \le C_0 (\alpha_{s_1}/\eps+\eps)$.
Since $y_\infty \in \ball(x_*, \delta_1)$, \eqref{max shift mode proof1} applies, so that we also have $\|\nabla f(y_\infty)\| \ge \lambda \|y_\infty - x_*\|$. Combining these two inequalities, and using the fact that $\alpha_{s_1} \le \alpha_s$, we conclude.
\end{proof}

\begin{lemma}
\label{lem:medoid max shift angle}
Let $(y_k)$ denote the medoid Max Shift sequence originating from some point $x_0$ in $\cU_{2s}$ for some level $s > 0$. 
At each step $k$, except perhaps for the last shift,
\begin{equation}
\label{medoid max shift angle}
y_{k+1} - y_k
= \|y_{k+1} - y_k\| N(y_k) \pm \frac{C_0 (\alpha_s/\eps+\eps)^{3/2}}{\|\nabla f(y_k)\|^{1/2}}, 
\end{equation}
whenever $\nabla f(y_k) \ne 0$.
In particular, medoid Max Shift satisfies \propertyref{3}.
\end{lemma}

\begin{proof}
Let $\alpha$ be short for $\alpha_s$.
Let $u_{k+1} := y_{k+1} - y_k$.
The result comes from comparing $f(y_{k+1})$, which by construction maximizes $f(y)$ over $y \in \cY \cap \ball(y_k, \eps)$, with $f(y'_{k+1})$ where $y'_{k+1}$ is a closest point to $x'_{k+1} : = y_k + (\eps-\alpha) N(y_k)$. The arguments around \eqref{medoid max shift level cluster proof1} apply to give that $\|y'_{k+1} - x'_{k+1}\| \le \alpha$ if $\eps$ is small enough. 
Hence, by the triangle inequality, $\|y'_{k+1} - y_k\| \le \eps$.
By construction of $y_{k+1}$, this implies that $f(y_{k+1}) \ge f(y'_{k+1})$.

Using \eqref{taylor2}, on the one hand we have
\begin{align}
f(y_{k+1})
\le f(y_k) + \nabla f(y_k)^\top u_{k+1} + \tfrac12 \kappa_2 \eps^2,
\end{align}
and on the other hand, with the help of \eqref{kappa1} and the definition of $\kappa_1$, we have
\begin{align}
f(y'_{k+1})
&\ge f(x'_{k+1}) - \kappa_1 \alpha \\
&\ge f(y_k) + (\eps-\alpha) \|\nabla f(y_k)\| - \tfrac12 \kappa_2 (\eps-\alpha)^2 - \kappa_1 \alpha \\
&\ge f(y_k) + \eps \|\nabla f(y_k)\| - C (\alpha+\eps^2).
\end{align}
We thus have 
\begin{align}
\nabla f(y_k)^\top u_{k+1} + \tfrac12 \kappa_2 \eps^2
\ge \eps \|\nabla f(y_k)\| - C_1 (\alpha+\eps^2),
\end{align}
implying
\begin{align}
N(y_k)^\top u_{k+1} 
&\ge \eps - C_2 (\alpha+\eps^2)/\|\nabla f(y_k)\| \\
&\ge \big(1 - C_2 (\alpha/\eps+\eps)/\|\nabla f(y_k)\|\big) \eps_{k+1},
\end{align}
with $\eps_{k+1} := \|u_{k+1}\|$. (Note that we used the fact that $\eps_{k+1} \le \eps$, by \lemref{medoid max shift step size}.) 
This is analogous to \eqref{max shift angle proof1} in the proof of \lemref{max shift angle}, and the remaining arguments are also analogous.
\end{proof}

We have thus established that medoid Max Shift satisfies Properties 1 and 3, and a weaker version of \propertyref{2}. But this is enough, by the same arguments underlying \thmref{prototype}, to show that, given a level $s > 0$, if $\alpha_s/\eps$ and $\eps$ are small enough, when initialized at any point in $\up_{2s}$ that is in the basin of attraction of some mode, medoid Max Shift converges to a medoid point within distance $C (\alpha_s/\eps+\eps)$ of that mode. 
We may thus state that, as $\alpha_s \to 0$ and $\eps \to 0$ in such a way that $\alpha_s/\eps \to 0$, the algorithm converges to that mode. In that sense, we may conclude that 
\begin{center}
{\em Medoid Max Shift is consistent}.
\end{center}

This is for the algorithm. We now discuss the consistency of the method, which is, as usual, defined by applying the algorithm to a KDE $\hat f_{n,h}$:
\begin{equation} \label{max shift method}
\hat x(0) = x_0; \quad
\hat x(k+1) \in \argmax_{y \in \cY \cap \ball(\hat x(k), \eps)} \hat f_{n,h}(y), \quad k \ge 0.
\end{equation}
The (uniform) consistency of this method can be established following the same line of arguments detailed in \secref{methods}, most specifically, \secref{max (slope) shift method}, from which we borrow some notation, and as we did there, we assume that $h\to 0$ and $\log n/(nh^{d+4}) \to 0$ as $n\to \infty$. In addition, we also require that $\alpha_{s_n}/\eps\to 0$ for a sequence of positive numbers $\{s_n\}$ tending to zero as $n\to\infty$. Then there is a measurable set $\Xi_{\eps,n,h}$ with probability measure tending to one as $\eps \to 0$ and $n\to\infty$, such that starting from any $x\in\Xi_{\eps,n,h}$, the sequence produced by the medoid Max Shift method converges to $\gamma_x(\infty)$. Details of proof are omitted.

\section{Discussion}
\label{sec:discussion}

Although we have covered a good amount of territory, there remain some interesting open questions in regards to the consistency of hill-climbing algorithms and methods.  

{\em Blurring Mean Shift.} 
While we have established the consistency of Mean Shift, the behavior of Blurring Mean Shift is not as well understood at the moment, even though some results do exist \citep{cheng1995, carreira2008generalised, carreira2006fast}. Although the Blurring Mean Shift is often seen as a faster version of Mean Shift, we have reasons to speculate that these two approaches are in fact quite distinct.

{\em Quick Shift.}
\citet{vedaldi2008quick} proposed the following hill-climbing algorithm. It is of medoid-type, with the sample being the default medoid set, as usual.
Assuming that the density is available, starting at an arbitrary medoid point, the algorithm iteratively moves to the closest medoid point within a certain neighborhood radius whose density is strictly larger than the current value, or in formula,
\begin{equation} \label{quick shift}
x(0) = x_0; \quad
x(k+1) \in \argmin \big\{\|y - x(k)\| : y \in \cY \cap \ball(x(k), \eps) \text{ with } f(y) > f(x(k))\big\}, \quad k \ge 0.
\end{equation}
A population analog of quick shift is not straightforward to define as, by continuity of the density, there is no closest point in the population within distance $h$ whose density value is strictly larger. 
\citet{NIPS2017_f457c545} showed that Quick Shift can be used for a variety of tasks, including finding density modes, and in subsequent work, \citet{jiang2018quickshift++} studied the consistency properties of this algorithm for the task of clustering. The latter is done for an initialization in some subset of the basin of attraction of a mode that, to quote the authors, ``satisfy the property that any path leaving [this region] must sufficiently decrease in density at some point". As is readily seen, in dimension $d \ge 2$, under the existence of a saddle point, this restricts the initialization to leaf clusters --- incidentally, the same restriction as in the study of Max Shift in \cite{chazal2013persistence}.
As far as we know, proving (or disproving) the consistency of Quick Shift over the entire basins of attraction remains an open problem.



\subsection*{Acknowledgments}
We are grateful to Melvin Leok for some pointers to the literature on ODEs.  
This work was partially supported by an NSF grant (DMS 1821154).

\bibliographystyle{chicago}
\bibliography{../ref}

\appendix

\section{Auxiliary results}
\label{sec:auxiliary}

\paragraph{Proof of \lemref{full continuity}}
\begin{proof}
We use the same notation as used in the proof of \thmref{prototype}, except that we make the dependence on the starting point $x$ explicit whenever needed as in, e.g., $t_\#(x)$ denoting $t_\#$ when associated with $x$. 
Let $\cA_*$ be the basin of attraction associated with a mode $x_*$, that is, $\cA_* = \{x\in\bbR^d: \lim_{t\to\infty} \gamma_x(t) = x_*\}$. 
With the notation $\cZ(x) := \zeta_x([0,\ell_x])$, where $\ell_x$ is the length of $\gamma_x$, we only need to show that for every $x\in\cA_*$, 
\begin{align}
\label{Z continuity}
\haus(\cZ(x), \cZ(y)) \to 0, \quad \text{as } \|y-x\| \to0.
\end{align}

By \lemref{cluster ball}, we can take $\delta_*>0$ small enough that there exist $\lambda>0$ and $c_*\ge1$ such that $\ball(x_*,c_*\delta_*)\subset\cA_*$, and all the eigenvalues of $\nabla^2 f(x)$ for all $x\in\ball(x_*,c_*\delta_*)$ are upper bounded by $-\lambda$, and $\inf_{x\in\bar\ball(x_*,\delta_*/3)} f(x) > \sup_{x\in\partial\ball(x_*,c_*\delta_*)} f(x)$.

For any point $x\in\cA_*\setminus \ball(x_*,\delta_*/3)$, define $t_\diamond(x) := \inf\{t \ge 0 : \|\zeta_x(t) - x_*\| = \delta_*/6\}$, and
$
\mu(x) := \frac12 \min\{\|\nabla f(z)\| : z \in \cZ_{t_\diamond(x)}(x)\}
$. Notice that $\|\nabla f(y)\|\ge \nu(x)$ for all $y\in\cS(x) :=\ball(\cZ_{t_\#(x)}(x), \delta_{\rm tube}^\prime)$, where $\delta_{\rm tube}^\prime:=\mu(x)/\kappa_2$ by \eqref{kappa2}. Similarly, by using \eqref{DN},
\begin{equation}
\sup_{y\in\cS(x)}\|DN(y)\| \leq \frac{\kappa_2}{\nu(x)} + \kappa_2 .
\end{equation}
%
Define $\cV(x):=\ball(\cZ_{t_\#(x)}(x), \delta_{\rm tube}^\prime/2)$, and notice that $\cS(x) = \ball(\cV(x), \delta_{\rm tube}^\prime/2)$. 
If $y,z \in \cV(x)$ are such that $\|y-z\| \le \delta_{\rm tube}^\prime/4$, then $z \in \ball(y, \delta_{\rm tube}^\prime/2)\subset\cS(x)$, and because that ball is convex, we have
\begin{equation}
\label{N_Lipschitz_1}
\|N(y) - N(z)\| 
\le \Big(\frac{\kappa_2}{\nu(x)} + \kappa_2\Big) \|y-z\|.
\end{equation}
If, on the other hand, $\|y-z\| > \delta_{\rm tube}^\prime/4$, then we can simply write
\begin{equation}
\label{N_Lipschitz_2}
\|N(y) - N(z)\| 
\le 2
= \frac{2}{\delta_{\rm tube}^\prime/4} \delta_{\rm tube}^\prime/4
\le \frac{8}{\delta_{\rm tube}^\prime} \|y-z\|.
\end{equation}
Hence, $N$ is $\kappa$-Lipschitz on $\cV(x)$, where $\kappa:= \max\{\kappa_2/\nu(x) + \kappa_2, 8/\delta_{\rm tube}^\prime\}$. 

Take a positive constant $\delta_\diamond \le \frac{1}{2}\exp\{-\kappa t_\diamond(x)\}\delta_{\rm tube}^\prime.$ 
Suppose that
\[
\cH(x) := \big\{\zeta_y(\tau): \tau \in [0,t_\diamond(x)], y \in \ball(x,\delta_\diamond)\big\}  \nsubset \cV(x).
\] 
Then there exist $y\in\ball(x,\delta_\diamond)$ and an escaping time $t_\square\in (0,t_\diamond(x))$ such that $\zeta_y(t_\square)\in\partial \cV(x)$ and $\{\zeta_y(\tau): \tau \in [0,t_\square)\} \subset \cV(x)$. This is impossible because applying a standard result on the dependence of the gradient flow on the initial condition, for example, the main theorem in \citep[Sec 17.3]{hirsch2012differential}, we have
\begin{equation} \label{hirsch_initial}
\|\zeta_x(t_\square) - \zeta_y(t_\square)\|
\le \|x-y\| \exp(\kappa t_\square) < \frac{1}{2} \delta_{\rm tube}^\prime,
\end{equation}
which would lead to $\zeta_y(t_\square) \in \cV(x)$, a contradiction against the definition of $t_\square$ as $\cV(x)$ is an open set. 
Therefore we must have $\cH(x) \subset \cV(x)$, and we can use the main theorem in \citep[Sec 17.3]{hirsch2012differential} to obtain
\begin{equation}
\label{initial_continuous}
\|\zeta_x(t) - \zeta_y(t)\|
\le \|x-y\| \exp(\kappa t) \le \frac{1}{2} \delta'_{\rm tube}, \quad \forall t \in [0,t_\diamond(x)], \quad \forall y \in \ball(x,\delta_\diamond).
\end{equation}

We further require that 
\begin{align}
\label{delta_diamond}
\delta_\diamond \le \frac{1}{6} \delta_* \exp(-\kappa t_\diamond(x)).
\end{align}
Then by the first inequality in \eqref{initial_continuous}, $\|\zeta_x(t_\diamond(x)) - \zeta_y(t_\diamond(x))\| \le \frac{1}{6}\delta_*$, and hence 
\[
\|\zeta_y(t_\diamond(x)) - x_*\| \le \|\zeta_x(t_\diamond(x)) - \zeta_y(t_\diamond(x))\| + \|\zeta_x(t_\diamond(x)) - x_*\| \le \frac{1}{3} \delta_*.
\]
This implies that $t_\#(y) \le t_\diamond(x)$, by definition of $t_\#(\cdot)$.

For $y \in \ball(x,\delta_\diamond)$, without loss of generality, suppose that $t_\#(y)\ge t_\#(x)$. Notice that 
\begin{equation}
\cZ_{t_\#(y)}(y) = \cZ_{t_\#(x)}(y) \,\bigcup\, \big\{\zeta_{y}(t): \; t\in[t_\#(x),t_\#(y)]\big\}.
\end{equation}
With notation $\phi = \sup_{t\in[t_\#(x),t_\#(y)]} \|\zeta_{y}(t) - \zeta_{y}(t_\#(x))\|$, we can write
\begin{align}
\label{haus_dist}
\haus(\cZ_{t_\#(x)}(x), \cZ_{t_\#(y)}(y)) 
&\le  \haus(\cZ_{t_\#(x)}(x), \cZ_{t_\#(x)}(y) ) + \phi \\
&\le  \sup_{t\in[0,t_\#(x)]}\|\zeta_x(t) - \zeta_{y}(t)\| + \phi\\
&\le  \|x-y\| \exp\{\kappa t_\#(x)\} + \phi, \label{haus_dist_last}
\end{align}
where the last equality is a result of \eqref{initial_continuous}.

For any $z\in\ball(x_*,c_*\delta_*)$ with $z \ne x_*$, using a Taylor expansion about $z$, we obtain
\begin{align}
0<f(x_*) - f(z) \le (x_* - z)^\top \nabla f(z) - \frac12 \lambda \|x-x_*\|^2.
\end{align}
Hence $(x_* - z)^\top \nabla f(z) > \frac12 \lambda \|z-x_*\|^2$. Denote $\xi_x(t) = \|\zeta_x(t) - x_*\|^2$. Then, for all $t\in[t_\#(x),\ell_x)$, where $\ell_x$ is the length of $\gamma_x$, we have $\zeta_x(t)\in\ball(x_*,c_*\delta_*)$, and
\begin{align}
\label{xidot}
\dot \xi_x(t) = 2(\zeta_x(t) - x_*)^\top \dot\zeta_x(t) = 2\frac{(\zeta_x(t) - x_*)^\top\nabla f(\zeta_x(t))}{\|\nabla f(\zeta_x(t))\|} 
< - \lambda \frac{\|\zeta_x(t) - x_*\|^2}{\|\nabla f(\zeta_x(t))\|} <0.
\end{align}
In other words, for all $t\in[t_\#(x),\ell_x)$, $\zeta_x(t)$ stays in $\ball(x_*,\delta_*/3)$ and its distance to $x_*$ strictly decreases as $t$ increases. Since $t_\#(y)\in[t_\#(x),t_\diamond(x)]$, we have
\begin{align}
\xi_x(t_\#(y)) & = \xi_x(t_\#(x)) + \int_{t_\#(x)}^{t_\#(y)} \dot \xi_x(s) \d s \\
& \le \frac{1}{9}\delta_*^2 - \lambda \int_{t_\#(x)}^{t_\#(y)} \frac{\|\zeta_x(s)-x_*\|^2}{\|\nabla f(\zeta_x(s))\|} \d s \\
& \le \frac{1}{9}\delta_*^2 -  \frac{\lambda}{\kappa_1} (\delta_*/3)^2 (t_\#(y)-t_\#(x)) \\
&=: \frac{1}{9}\delta_*^2 -  C_1(t_\#(y)-t_\#(x)).
\end{align}
Noticing that by \eqref{initial_continuous} and \eqref{delta_diamond},
\begin{align}
\|\zeta_y(t_\#(y)) - \zeta_x(t_\#(y))\| \le \frac{1}{6}\delta_*  < \frac{1}{3}\delta_* = \|\zeta_y(t_\#(y)) - x_*\|, 
\end{align}
and using the triangle inequality, we have
\begin{align}
\xi_x(t_\#(y)) &= \|\zeta_x(t_\#(y)) - x_*\|^2 \\
&\ge \Big( \|\zeta_y(t_\#(y)) - x_*\| - \|\zeta_y(t_\#(y)) - \zeta_x(t_\#(y))\| \Big)^2 \\
&\ge \Big( \delta_*/3 - \|x-y\| \exp(\kappa t_\#(y)) \Big)^2 \\
& = \frac{1}{9}\delta_*^2 - \frac{2}{3}\delta_* \exp ( \kappa t_\#(y)) \|x-y\| + \exp(2\kappa t_\#(y)) \|x-y\|^2\\
& \ge \frac{1}{9}\delta_*^2 - \Big\{\frac{2}{3}\delta_* \exp ( \kappa t_\diamond(x)) - \delta_\diamond \exp(2 \kappa t_\#(x))\Big\}\|x-y\| \\
& = : \frac{1}{9}\delta_*^2 - C_2(x) \|x-y\|.
\end{align}
Note that by \eqref{delta_diamond},
\begin{align}
C_2(x) \ge 4 \delta_\diamond \exp(2 \kappa t_\diamond(x)) - \delta_\diamond \exp(2 \kappa t_\#(x)) >0.
\end{align}
We thus obtain $t_\#(y)-t_\#(x) \le (C_2(x)/C_1) \|x-y\|$. 
Hence, using the fact that $\zeta_y$ is parameterized by arc length,
\begin{align}
\phi
&= \sup_{t\in[t_\#(x),t_\#(y)]} \Big\|\int_{t_\#(x)}^t \dot\zeta_y(t) \d t \Big\| \\
&\le t_\#(y) - t_\#(x) \\
&\le (C_2(x)/C_1) \|x-y\|.
\end{align}

Combining this with \eqref{haus_dist_last}, we obtain
\begin{equation}
\label{hausdorff_continuity}
\haus(\cZ_{t_\#}(x), \cZ_{t_\#}(y)) \le C_3(x)\|x-y\|,
\end{equation}
where $C_3(x) = \exp\{\kappa t_\#(x)\} + C_2(x)/C_1$. 

As shown in \eqref{xidot}, once a gradient flow enters $\ball(x_*,\delta_*/3)$, it never escapes from this ball. Using \eqref{hausdorff_continuity}, we have $\haus(\cZ(x), \cZ(y))\le \delta_*$ when $\|x-y\|$ is small enough. Since $\delta_*>0$ can be made arbitrarily small, we then obtain \eqref{Z continuity}, which gives the conclusion of the lemma.
\end{proof}

\begin{lemma}
\label{lem:nu_continuity}
$\nu$ is a continuous function on $\cA$.
\end{lemma}

\begin{proof}
Noticing that $\cA_*$ is an open set, and $\cA$ is the union of all the $\cA_*$'s, we only need to show $\nu$ is continuous $\cA_*$. Without loss of generality, we again assume that $\delta_*$ in Property 2 is small enough that the same assumption on $\delta_*$ in the proof of \lemref{full continuity} holds.

Note that $\nu(x) = \frac{1}{2}\|\nabla f(x)\|$ for any $x\in\bar\ball(x_*,\delta_*/3)$. It is clear that $\nu$ is continuous on $\bar\ball(x_*,\delta_*/3)$, and thus we only need to show that $\nu$ is continuous on $\cA_*\setminus \ball(x_*,\delta_*/3)$. 

By denoting $d_\nabla(\cA \mid \cB) = \sup_{v\in \cA} \; \inf_{w\in \cB} | \|\nabla f(v)\| - \|\nabla f(w)\| |$, it follows from \eqref{hausdorff_continuity} that
\begin{align}
&|\nu(x) - \nu(y)| \\
&\le  \max\{ d_\nabla(\cZ_{t_\#}(x) \mid \cZ_{t_\#}(y)),\; d_\nabla(\cZ_{t_\#}(y) \mid \cZ_{t_\#}(x)) \} \\
&\le  \kappa_2 \haus(\cZ_{t_\#}(x), \cZ_{t_\#}(y)) \\
&\le  C_4(x)\|x-y\|,
\end{align}
where $C_4(x)= \kappa_2C_3(x)$. We have thus shown that $\nu$ is a continuous function on $\cA$. 
\end{proof}

\end{document}